\documentclass[11pt]{amsart}
\title{Connecting conformal dimension and Poincar\'e profiles}
\author{David Hume}
\address{School of Mathematics, University of Birmingham, Edgbaston, UK.}
\email{d.hume@bham.ac.uk}
\author{John M. Mackay}
\address{School of Mathematics, University of Bristol, Bristol, UK.}
\email{john.mackay@bristol.ac.uk}
\date{\today}

\usepackage{amsmath,amsthm,amsfonts,graphicx,amssymb}
\usepackage{hyperref}
\usepackage{xcolor}
\usepackage{tikz}

\numberwithin{equation}{section}
\newtheorem{theorem}[equation]{Theorem}
\newtheorem{proposition}[equation]{Proposition}
\newtheorem{corollary}[equation]{Corollary}
\newtheorem{lemma}[equation]{Lemma}

\newtheorem{question}[equation]{Question}

\newtheorem{definition}[equation]{Definition}
\newtheorem{remark}[equation]{Remark}

\newtheorem*{theorem*}{Theorem}

\newtheoremstyle{citing}
  {3pt}
  {3pt}
  {\itshape}
  {}
  {\bfseries}
  {}
  {.5em}
  {\thmnote{#3}}

\theoremstyle{citing}
\newtheorem*{varthm}{}

\DeclareMathOperator{\diam}{diam}

\DeclareMathOperator{\cut}{cut}
\DeclareMathOperator{\sep}{sep}

\DeclareMathOperator{\CAT}{CAT}

\DeclareMathOperator{\GL}{GL}
\DeclareMathOperator{\Cone}{Con}
\DeclareMathOperator{\St}{St}

\DeclareMathOperator{\Confdim}{Confdim}

\newcommand{\bdry}{\partial_\infty}

\newcommand{\setcon}[2]{\left\{#1\ \left|\ #2\right.\right\}}

\newcommand{\cC}{\mathcal{C}}

\newcommand{\R}{\mathbb{R}}
\newcommand{\C}{\mathbb{C}}

\newcommand{\N}{\mathbb{N}}
\newcommand{\Z}{\mathbb{Z}}
\newcommand{\HH}{\mathbb{H}}
\newcommand{\ba}{\mathbf{a}}
\newcommand{\bbb}{\mathbf{b}}
\newcommand{\bc}{\mathbf{c}}

\newcommand{\pcrit}{p_{\Lambda}}

\def\XXint#1#2#3{{\setbox0=\hbox{$#1{#2#3}{\int}$}
\vcenter{\hbox{$#2#3$}}\kern-.5\wd0}}

\numberwithin{equation}{section}

\begin{document}

\begin{abstract} 
	We strengthen the connection between the Ahlfors-regular (AR) conformal dimension $\Confdim(Z)$ of a compact AR metric space $Z$ and a certain critical exponent of the Poincar\'e profiles $\pcrit$ of its hyperbolic cone $X$ in the sense of Bonk--Schramm. We prove that the two values are equal in two situations: firstly, when $Z$ is a product $C\times [0,1]$ where $C$ is a compact AR metric space; and secondly when $X$ is quasi-isometric to a Heintze manifold $\R^n\rtimes_A\R$ where $A\in\GL(n,\R)$ is diagonalisable.  A key tool is a lower bound for $\pcrit$ for combinatorial round trees which also applies to various random group models and families of Coxeter groups.

	We also show that for a torsion free hyperbolic group $G$, $\pcrit(G)>1$ if and only if Benjamini--Schramm--Tim\'ar's separation profile grows faster than $r^\alpha$ for some $\alpha>0$, if and only if $\Confdim(\bdry G)>1$.  On the other hand, we find new, non-virtually-Fuchsian examples of groups with the same separation profile as $\mathbb{H}^2$.
	All these results imply various obstructions to coarse and regular embeddings of such groups.  
\end{abstract}

\maketitle

\section{Introduction}
\label{sec:intro}

There is a classical correspondence between isometries of hyperbolic spaces and M\"obius transformations of their sphere at infinity.
This was generalized by Paulin and Bonk--Schramm to give a correspondence between quasi-isometries of Gromov hyperbolic spaces and quasisymmetric homeomorphisms of their boundaries, making it natural to seek to relate the large-scale geometric properties of Gromov hyperbolic spaces with analytic properties of their boundaries and vice-versa.

One highly studied quasisymmetry invariant of compact metric spaces is Pansu's conformal dimension \cite{Pan-89-cdim}. In previous work of the authors and Romain Tessera, 
we demonstrated that a critical exponent of ``Poincar\'e profiles'', that is,  Poincar\'e inequalities on bounded subsets, equals the conformal dimension of the boundary at infinity for certain Gromov hyperbolic spaces including classical hyperbolic spaces~\cite{HumeMackTess-Pprof, HumeMackTess-PprofLie}. There are two key benefits of working with the critical exponent of Poincar\'e profiles: it is defined for all proper, geodesic metric spaces; and, in addition to being invariant under quasi-isometries, it behaves monotonically with respect to a much broader class of maps called regular maps, which include coarse embeddings, despite the fact that these maps typically do not have good boundary extensions.

The goal of this paper is to further develop our understanding of the relationship between this critical exponent and conformal dimension, in particular determining situations where the two values agree, and where existing bounds on conformal dimension can be reproved using Poincar\'e profiles.

As many lower bounds on conformal dimension are proved by finding embedded ``round trees'', we start by studying their Poincar\'e profiles. We use this to determine many situations where the conformal dimension of the boundary of a hyperbolic space agrees with a critical exponent for the Poincar\'e profiles, with consequent non-embedding results.

We also establish two new phenomena for hyperbolic groups around conformal dimension $1$.  First, we use (irregular) round trees to show that torsion-free hyperbolic groups with conformal dimension $>1$ have $L^1$-Poin\-ca\-r\'e profiles, i.e.\ separation profiles \cite{BenSchTim-12-separation-graphs}, which grow $\gtrsim r^{\epsilon}$ for some $\epsilon>0$.
Second, we show that many hyperbolic groups whose boundaries have conformal dimension $1$, but do not attain this value, nevertheless have the same Poincar\'e profiles as the hyperbolic plane.

Let us now describe our results in more detail.
\subsection{Main concepts}
We recall the definition of the Poincar\'e profiles of a graph.
\begin{definition}[{\cite[\S 1.1]{HumeMackTess-Pprof}}]
	\label{def:poinc-profile}
	For a finite graph $\Gamma$ with vertex set $V(\Gamma)$ and edge set $E(\Gamma)$, for $p \in [1,\infty]$ its \emph{$L^p$-Poincar\'e constant} is
	\[
		h^p(\Gamma) = \inf \left\{\left. \frac{\|\nabla f\|_p}{\|f\|_p} \ \right|\ f:V(\Gamma) \to \R, \sum_{v \in V(\Gamma)} f(v)=0, f \not\equiv 0\right\}.
	\]
	Here, $|\nabla f|:E(\Gamma)\to \R$ is defined by $|\nabla f|(e) = |f(x)-f(y)|$ where $e \in E(\Gamma)$ has endpoints $x,y \in V(\Gamma)$.

	For an infinite bounded degree graph $X$, and $p \in [1,\infty]$, its \emph{$L^p$-Poincar\'e profile} $\Lambda_X^p:\N\to\R$ is
	\[
		\Lambda_X^p(r) = \sup\{|\Gamma|\, h^p(\Gamma) : \Gamma\leq X, |\Gamma|\leq r\},
	\]
	where $|\Gamma|:=|V(\Gamma)|$. 
\end{definition}

We consider functions on $\N$ up to the equivalence relation $\simeq$, where $f \lesssim g$ if there exists $C>0$ so that $f(r) \leq Cg(Cr+C)+C$ for all $r$, and $f \simeq g$ if $f \lesssim g$ and $g \lesssim f$.

A function $f:VX \to VY$ is a \emph{regular map} if there exists $C>0$ so that $f$ is $C$-Lipschitz and for all $y \in VY$, $|f^{-1}(y)| \leq C$.  We say $X$ \emph{regularly embeds} in $Y$ if there exists a regular map from $X$ to $Y$.
In particular, coarse embeddings between bounded degree graphs are regular maps, as is the inclusion map for a finitely generated subgroup of a finitely generated group.

If there exists a regular map $f: VX \to VY$ then $\Lambda_X^p \lesssim \Lambda_Y^p$ for all $p \in [1,\infty]$ \cite{HumeMackTess-Pprof}.
Thus, $\Lambda_X^p$ is a quasi-isometry invariant of $X$ for each $p \in [1,\infty]$.  In particular, if $X$ is any metric space quasi-isometric to a bounded degree graph $Y$, we can set $\Lambda_X^p := \Lambda_Y^p$ and then $\Lambda_X^p$ is well-defined up to $\simeq$, independently of the choice of $Y$.

For $p=1$, $\Lambda_X^1 \simeq \sep_X$, where $\sep_X$ is the separation profile of Benjamini--Schramm--Tim\'ar~\cite{BenSchTim-12-separation-graphs}.

Let us now turn to conformal dimension.
A metric space $Z$ is Ahlfors $Q$-regular if there exists a Borel regular measure $\mu$ on $Z$ and a constant $C>0$ so that for any $z \in Z$, $0<r\leq \diam(Z)$, $\frac{1}{C} r^Q \leq \mu(B(z,r)) \leq Cr^Q$.
The \emph{(Ahlfors regular) conformal dimension} of $Z$, denoted $\Confdim(Z)$, is the infimum of all $Q$ such that $Z$ is quasisymmetric to an Ahlfors $Q$-regular space.
As the conformal dimension is a quasisymmetry invariant, for a Gromov hyperbolic space $X$, $\Confdim(\bdry X)$ is well-defined and depends only on the quasi-isometry class of $X$.
It can be more sensitive than the topology of the boundary. For example, $\HH_\R^4$ and $\HH_\C^2$ are not quasi-isometric as $\Confdim(\bdry \HH_\R^4)=3$ and $\Confdim(\bdry \HH_\C^2)=4$, even though $\bdry \HH_\R^4$ and $\bdry \HH_\C^2$ are homeomorphic.
An important example for us is that when $Z$ is a compact Ahlfors $Q$-regular space, $\Confdim(Z \times [0,1])=Q+1$, see~\cite[Example 4.1.9]{Mac-Tys-cdimexpo} and references therein.

In earlier work of the authors and Romain Tessera, the Poincar\'e profiles of each rank-one symmetric space $X$ were computed, in particular, for $p<Q:=\Confdim(\bdry X)$ we have $\Lambda_X^p(r) \simeq r^{1-\frac{1}{Q}}$, $\Lambda_X^Q(r) \simeq r^{1-\frac{1}{Q}}\log^{\frac{1}{Q}}(r)$, while for $p>Q$ we have $\Lambda_X^p(r) \simeq r^{1-\frac{1}{p}}$ \cite[Theorem 12]{HumeMackTess-Pprof}.
The lower bound in the case $p>Q$ comes from an embedded $3$-regular tree $T_3$ which has $\Lambda_{T_3}^p(r) \simeq r^{1-\frac{1}{p}}$ for all $p \in [1,\infty)$. This threshold, beyond which trees are the critical obstruction to Poincar\'e inequalities, motivates the following definition.

\begin{definition}\label{def:profile-crit}
	Given a graph $X$, we define its \emph{critical exponent for Poincar\'e profiles} $\pcrit(X)$ as  
\begin{equation}\label{eq:pcrit}
    \pcrit(X) := \inf \{ p \geq 1 : \Lambda^p_X(r) \lesssim r^{1-\frac{1}{p}} \}.
\end{equation}
\end{definition}
This critical exponent satisfies the following monotonicity property amongst bounded degree graphs: $\pcrit(X) \leq \pcrit(Y)$ whenever $X$ regularly embeds into $Y$.  Moreover, the set in equation \eqref{eq:pcrit} is always connected, see Corollary~\ref{cor:crit-exponent-interval}. We note that these results do not assume hyperbolicity, and indeed, the critical exponent can be useful in other scenarios; for instance, the critical exponent of any finitely generated nilpotent group with polynomial growth of degree $d$ is exactly $1-\frac1d$ \cite{HumeMackTess-Pprof}.

In the definition here, for graphs containing a regularly embedded copy of $T_3$ (such as Cayley graphs of non-elementary hyperbolic groups) we could equivalently have taken ``$\simeq r^{1-\frac{1}{p}}$'', see Corollary~\ref{cor:crit-exponent-interval}.

\subsection{Results on critical exponents}
A goal of this paper is to progress our understanding of the following question.  Recall that for any compact Ahlfors regular space $Z$ there is a geodesic Gromov hyperbolic graph $\Cone(Z)$ so that $Z$ is a boundary of $\Cone(Z)$~\cite{BonkSchramm}.  If $G$ is a Gromov hyperbolic group, $\Cone(\bdry G)$ is quasi-isometric to $G$.

\begin{question}[{\cite[Question 8.6]{HumeMackTess-PprofLie}}]\label{quequality}
    For which compact Ahlfors regular metric spaces $Z$ does the equality $\pcrit(\Cone(Z)) = \Confdim(Z)$ hold?  Is there equality for every boundary of a hyperbolic group?
\end{question}
A positive answer to this question for hyperbolic groups would imply a positive answer to the following question.
\begin{question}\label{qu:regmap-confdim-monotone}
If a hyperbolic group $G$ regularly embeds in a hyperbolic group $H$, then $\Confdim(\bdry G) \leq \Confdim(\bdry H)$.	
\end{question}

We always have one inequality in Question~\ref{quequality}.
\begin{theorem}[{\cite[Theorem 5.16]{HumeMackTess-PprofLie}}]\label{thm:crit-exponent-previous-upper-bound}
    Let $Z$ be a compact Ahlfors regular metric space and let $\Cone(Z)$ be the hyperbolic cone of $Z$. Then
    \[
        \pcrit(\Cone(Z)) \leq \Confdim(Z).
    \]
\end{theorem}

As already mentioned, the first class of examples with a positive answer to Question~\ref{quequality} are provided by previous work of the authors with Tessera, combined with a result of Bonk and Kleiner. 

\begin{theorem}\label{thm:hyp-attain-confdim-critexp}
    If $X$ is quasi-isometric to a non-elementary Gromov hyperbolic group whose boundary attains its Ahlfors regular conformal dimension, e.g., a rank-one symmetric space or a Bourdon Fuchsian building, then $\pcrit(X) = \Confdim(\bdry X)$.  
\end{theorem}
    
\begin{proof}
	For $p > \Confdim(\bdry X)$, $\Lambda^p_X(r) \lesssim r^{1-\frac{1}{p}}$ by \cite[Theorem 6.1]{HumeMackTess-PprofLie}.
    If $\Confdim(\bdry X)=1$ then we are done.
	Otherwise, by \cite[Theorem 1.3]{BK-05-cdim-cannon} if $X$ attains its Ahlfors regular conformal dimension $Q>1$ then $\bdry X$ admits a $(1,Q)$-Poincar\'e inequality in the sense of Heinonen--Koskela, so 
	$\Lambda^Q_X(r) \gtrsim r^{1-1/Q}\log(r)^{1/Q}$ by \cite[Theorem 11.3]{HumeMackTess-Pprof}.
\end{proof}

Spaces of the form $Z \times [0,1]$ are a classic family of metric spaces which attain their conformal dimension but are quite different to the boundaries of Theorem~\ref{thm:hyp-attain-confdim-critexp}, for example they rarely admit Poincar\'e inequalities.

We also find a positive answer to Question~\ref{quequality} for these spaces.
(We slightly simplify statements of results in this introduction.)
\begin{theorem}[Theorem \ref{thm:product-crit-exponent}]
	\label{thm:product-crit-exponent-intro}
    Let $Z$ be a compact Ahlfors regular space.
	Then
    \[
    \pcrit(\Cone(Z\times [0,1])) = \Confdim(Z\times[0,1]).
    \]
\end{theorem}
We deduce this from the special case (Theorem~\ref{thm:pcritRT}) when $Z \times [0,1]$ is the boundary of a ``combinatorial round tree'', a variation of a construction of Gromov introduced in \cite{Mac-12-random-cdim}.  
(See Stark~\cite{Stark-25-round-tree-survey} for a recent survey.)

We apply this to give a positive answer to Question \ref{quequality} for a certain class of Heintze groups, which also partially answers \cite[Question 8.7]{HumeMackTess-PprofLie}.
\begin{corollary}[Corollary~\ref{cor:crit-exponent-heintze}]
	\label{cor:crit-exponent-heintze-intro}
    Let $G=\R^n\rtimes_A\R$ where $A\in\GL(n,\R)$ is diagonalisable with positive eigenvalues. Then
    \[
    \pcrit(G) = \Confdim(\bdry G).
    \]
\end{corollary}

Combinatorial round trees also feature in the following contexts, where they are used to give lower bounds on conformal dimension.
\begin{itemize}
	\item boundaries of random groups in the few relator, polynomially many relator, and Gromov density models~\cite{Mac-12-random-cdim,Mac-16-confdim-subcplxs,Frost-22-confdim-random}, see Section~\ref{ssec:random-groups};
	\item boundaries of hyperbolic Coxeter groups with underlying graph a complete graph and edge labels uniformly bounded, getting infinitely many quasi-isometry classes of hyperbolic groups with Pontryagin sphere boundary~\cite{Field-Gupta-Lyman-Stark-25-coxeter-confdim}, see Section~\ref{ssec:coxeter-application};
	\item boundaries of hyperbolic right-angled Coxeter groups (RACG), getting infinitely many quasi-isometry classes of RACG with Pontryagin sphere boundary, and infinitely many quasi-isometry classes of RACG that virtually algebraically fibre for each prescribed virtual cohomological dimension $n \geq 2$~\cite{Cashen-Dani-Schreve-Stark-25-RACG-confdim}, see Section~\ref{ssec:coxeter-application}
\end{itemize}

Thanks to Theorem~\ref{thm:pcritRT}, in each such situation we immediately get the same lower bounds for $p_\Lambda$. 
All these bounds on critical exponents have consequences for non-em\-bed\-ding results.
For example, a Heintze group $G=\R^n \rtimes_A \R$ with diagonal $A$ does not regularly/coarsely embed into a hyperbolic group of conformal dimension $< \Confdim(\bdry G)$.  And, in Gromov's density model for any $d \in (0,\frac{1}{8})$, random groups at large $\ell$ do not regularly embed into random groups with small $\ell$.

There are also non-hyperbolic groups with finite $\pcrit$, for example when $X$ is a finitely generated nilpotent group, $\pcrit(X)$ is the polynomial growth degree of $X$ \cite{HumeMackTess-Pprof}. We raise the following:

\begin{question} Find a finitely generated group $G$ with $\pcrit(G)<+\infty$ which 
\begin{enumerate}
    \item[(i)] is not a subgroup of a relatively hyperbolic group with virtually nilpotent peripheral subgroups; or more generally
    \item[(ii)] does not regularly embed into some real hyperbolic space.
\end{enumerate}
\end{question}

\subsection{Results at conformal dimension $1$}
We are able to use techniques from \cite{Mac-10-confdim} to build more irregular round trees to show that hyperbolic groups which do not split over finite or virtually cyclic subgroups have separation profiles larger than some positive power.  Note that an infinite hyperbolic group $G$ does not split over a finite or virtually cyclic subgroup if and only if $\bdry G$ is connected, locally connected and has no local cut points. Recall that Benjamini-Schramm-Tim\'ar's separation profile $\sep_X$ satisfies $\sep_X \simeq \Lambda_X^1$ by \cite{HumeMackTess-Pprof}; we stick to the latter notation in this introduction.

\begin{theorem}\label{thm:power-sep-hyp-no-local-cut-points}
	If $G$ is an infinite hyperbolic group that doesn't split over a finite or virtually cyclic subgroup, either $G$ is virtually co-compact Fuchsian (and so $\Lambda^1_G(r) \simeq \log(r)$) or $\Lambda^1_G(r) \gtrsim r^\alpha$ for some $\alpha>0$.
\end{theorem}
\begin{remark}
    An alternative, independent proof of Theorem~\ref{thm:power-sep-hyp-no-local-cut-points} has been announced by Bensaid--Genevois--Tessera~\cite{BenGenTes}.
\end{remark}
As a consequence, we strengthen Lazarovich--Le Coz \cite[Corollary 1.2]{Lazarovich-LeCoz-25} which shows the following result under the stronger assumption $\Lambda^1_G(r) \lesssim \log r$.
\begin{corollary}\label{cor:sep-less-than-powers-splits}
	If $G$ is a non-elementary hyperbolic group and $\Lambda^1_G(r) \not\gtrsim r^\alpha$ for any $\alpha>0$, then $G$ is virtually co-compact Fuchsian or splits over a finite or virtually cyclic subgroup.
\end{corollary}
Combining this with existing structural results on conformal dimension, we have the following trichotomy.
\begin{corollary}\label{cor:hyp-trichotomy} Let $G$ be an infinite hyperbolic group, virtually with no $2$-torsion,  with boundary $\bdry G$. Exactly one of the following families of equivalent statements holds:
\begin{enumerate}
    \item $\Confdim(\bdry G)=0$; $\Lambda^1_G(r)\simeq 1$; $G$ is virtually free.
	\item $\Confdim(\bdry G)=1$; $\log(r)\lesssim \Lambda^1_G(r)\lesssim r^\alpha$ for every $\alpha>0$; $G$ decomposes as an iterated amalgam of finite and Fuchsian groups over finite and virtually cyclic subgroups but is not virtually free.
	\item $\Confdim(\bdry G)>1$; $\pcrit(G)>1$; $\Lambda^1_G(r)\gtrsim r^\alpha$ for some $\alpha>0$; $G$ has an infinite quasi-convex subgroup which is not virtually Fuchsian and does not split over finite or virtually cyclic subgroups.
\end{enumerate}
\end{corollary}
\begin{proof}
	Follows from Theorem~\ref{thm:power-sep-hyp-no-local-cut-points}, \cite[Corollary 1.2]{CarMac-22-confdim1} (see \cite[Corollary 1.3]{Mac-10-confdim}), \cite[Theorem 6.1]{HumeMackTess-PprofLie} (case $d=0$).
\end{proof}
We also have a non-embedding consequence, as separation profiles are monotone under coarse embeddings.
\begin{corollary}\label{cor:non-embed-confdim-close-to-1}
	If $G$ is a hyperbolic group (virtually with no $2$-torsion) having conformal dimension $>1$, then $\Lambda^1_G(r) \gtrsim r^{\alpha}$ for some $\alpha\in (0,1)$, hence $G$ does not coarsely (or regularly) embed into any hyperbolic group with conformal dimension $< 1/(1-\alpha)$.
\end{corollary}
\begin{proof}
	By \cite[Theorem 6.1]{HumeMackTess-PprofLie}, if $H$ is hyperbolic with $\Confdim(\bdry H) < \frac{1}{1-\alpha}$, then there exists some $\varepsilon>0$ such that $\Lambda_H^1(r) \lesssim r^{\alpha-\varepsilon}$, so $G$ does not coarsely (or regularly) embed into $H$.
\end{proof}
This is evidence towards a positive answer to Question~\ref{qu:regmap-confdim-monotone}.

Beyond recognising whether conformal dimension is equal to $1$ or greater, calculating the $1$-Poincar\'e profiles of groups of conformal dimension $1$ is independently interesting. It is known that the hyperbolic plane has logarithmic $1$-Poincar\'e profile (or equivalently, separation profile), while Lazarovich--Le Coz construct a hyperbolic group of conformal dimension one whose $1$-Poincar\'e profile is superlogarithmic \cite{Lazarovich-LeCoz-25}. We may then ask for a classification.
\begin{question}\label{qu:logsep}
	Classify which hyperbolic groups $G$ have separation $\Lambda_G^1(r) \simeq \log(r)$.
Are these the groups with well-separated boundaries (see~\cite{Carrasco14-WS-hyp-confdim1})?
\end{question}
As a partial step towards this goal we add many new groups to this class.

We state the following theorem in terms of JSJ decompositions in the sense of Guirardel--Levitt~\cite{GuirardelLevitt-17-JSJ}, however, we do not need to work directly with the definition in this paper.
\begin{theorem}\label{thm:logsep}
	Let $G$ be a one-ended hyperbolic group which has a JSJ decomposition over $2$-ended subgroups consisting only of quadratically hanging and cylindrical vertex groups.
	Then $\Lambda^1_{G}(r) \simeq \log(r)$.
\end{theorem}
	An example of such a group is shown in Figure~\ref{fig:glue-three-punctured-tori-on-boundary}, and indeed this is typical: up to taking finite index subgroups, the groups considered in Theorem~\ref{thm:logsep} are exactly the fundamental groups of spaces found by gluing a collection of compact hyperbolic surfaces together along totally geodesic boundary components.

	This example further shows that the separation profile does not detect the attainment of conformal dimension $1$: $\pi_1(X)$ does not attain its conformal dimension of $1$ (see discussion and references in \cite{CarMac-22-confdim1}), but it has the same separation profile as groups which do, namely all the virtually co-compact Fuchsian groups.
    (Note that these groups cannot be distinguished by their other Poincar\'e profiles: any hyperbolic group $G$ with $\Confdim(\bdry G)=1$ has a quasi-isometrically embedded $3$-regular tree and so satisfies $\Lambda_G^p(r) \simeq r^{1-\frac1p}$ for all $p>1$ by \cite[Theorem 6.1]{HumeMackTess-PprofLie}.)

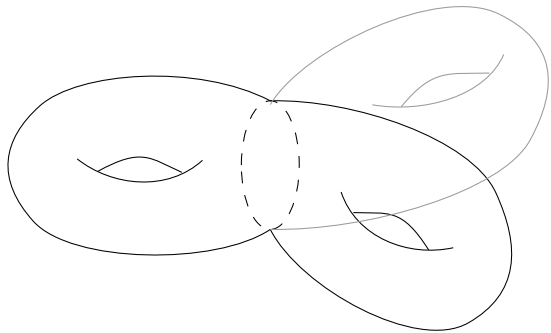
\begin{figure} 
\begin{tikzpicture}[x=0.75pt,y=0.75pt,yscale=-.8,xscale=.8]
\draw  [dash pattern={on 4.5pt off 4.5pt}] (294.5,138.57) .. controls (294.5,116.16) and (302.67,98) .. (312.75,98) .. controls (322.83,98) and (331,116.16) .. (331,138.57) .. controls (331,160.97) and (322.83,179.13) .. (312.75,179.13) .. controls (302.67,179.13) and (294.5,160.97) .. (294.5,138.57) -- cycle ;
\draw    (312.75,98) .. controls (267,74.13) and (191.5,79.63) .. (167.5,101.13) .. controls (143.5,122.63) and (139.5,147.63) .. (163.5,174.13) .. controls (187.5,200.63) and (278.5,202.13) .. (312.75,179.13) ;
\draw    (191,134.63) .. controls (213,152.13) and (248,155.63) .. (270,135.5) ;
\draw    (204,142.5) .. controls (232,127.13) and (237,134.13) .. (257,143.13) ;
\draw    (312.75,98) .. controls (365.46,96.67) and (439.5,122.13) .. (455,156.13) .. controls (470.5,190.13) and (470,219.63) .. (437.5,238.13) .. controls (405,256.63) and (331.61,216.09) .. (312.75,179.13) ;
\draw    (428.1,190.68) .. controls (399.72,196.92) and (366.83,184.23) .. (357.41,155.51) ;
\draw    (412.57,192.07) .. controls (395.42,165.05) and (387.45,169.27) .. (365.13,168.52) ;
\draw [color={rgb, 255:red, 155; green, 155; blue, 155 }  ,draw opacity=1 ]   (312.75,179.09) .. controls (373.46,180.37) and (458.75,155.65) .. (476.6,122.64) .. controls (494.45,89.63) and (493.88,60.99) .. (456.44,43.03) .. controls (419.01,25.07) and (334.48,64.43) .. (312.75,100.31) ;
\draw [color={rgb, 255:red, 155; green, 155; blue, 155 }  ,draw opacity=1 ]   (377.08,100.51) .. controls (410.33,106.14) and (448.87,94.68) .. (459.9,68.73) ;
\draw [color={rgb, 255:red, 155; green, 155; blue, 155 }  ,draw opacity=1 ]   (395.27,101.76) .. controls (415.37,77.35) and (424.71,81.16) .. (450.86,80.49) ;
\end{tikzpicture}
	\caption{A space $X$ with $\Lambda_{\pi_1(X)}^1(r)\simeq\sep_{\pi_1(X)}(r) \simeq \log(r)$.}	
	\label{fig:glue-three-punctured-tori-on-boundary}
\end{figure}

\subsection{Outline}
In Section~\ref{sec:crit-exp} we clarify the role of the critical exponent for Poincar\'e profiles.  In Sections~\ref{sec:round-trees} and \ref{sec:reg-round-trees} we introduce round trees and compute non-trivial bounds on their Poincar\'e profiles, which are applied to various examples in Section~\ref{sec:applications-reg-round-trees}.  In Section~\ref{sec:no-local-cut-points} we prove Theorem~\ref{thm:power-sep-hyp-no-local-cut-points} for hyperbolic groups whose boundaries have no local cut points, and in Section~\ref{sec:logupper} we prove Theorem~\ref{thm:logsep} giving new examples of groups with logarithmic separation.

\subsection{Notation}
As mentioned above, for functions $f, g$ we write $f \lesssim g$ if there exists $C>0$ so that $f(r) \leq Cg(Cr+C)+C$ for all $r$, and $f \simeq g$ if $f \lesssim g$ and $g \lesssim f$.
For any $x,y \in \R$ we write $x \preceq y$ if there exists $C>0$ (independent of relevant parameters) so that $x \leq Cy$, and $x \asymp y$ if $x \preceq y$ and $y\preceq x$.

\subsection*{Acknowledgements}
The first author was supported by the EPSRC grant EP/V027360/1 ``Coarse geometry of groups and spaces''. The second author thanks Chris Cashen and Sam Shepherd for helpful comments.

\section{Critical exponents}
\label{sec:crit-exp}

We collect some auxilliary observations which clarify the importance of the critical exponent for Poincar\'e profiles (Definition~\ref{def:profile-crit}).
\begin{lemma}\label{lem:lplq-bounds}
    For any vector $\ba \in \R^r$, and any $q \geq p \geq 1$,
	$\| \ba \|_q \leq \| \ba \|_p \leq r^{\frac{1}{p}-\frac{1}{q}} \| \ba \|_q$.
\end{lemma} 
\begin{proof}
  Write $\ba = (a_i)$.  We may assume all $a_i \geq 0$.  Then
  \begin{align*}
	\left(\sum a_i^q \right)^{1/q}
	& \leq \left( \left[\max a_i^{q-p}\right] \left[\sum a_i^p \right]\right)^{1/q} 
	\\ & \leq \left( \left[\sum a_i^p \right]^{(q-p)/p + 1}\right)^{1/q}
	 = \left(\sum a_i^p \right)^{1/p}
	  \\ & \leq \left( \left[\sum a_i^{pq/p}\right]^{p/q} r^{1-p/q}\right)^{1/p}
	 = \left(\sum a_i^q \right)^{1/q} r^{1/p-1/q}. \qedhere
  \end{align*}
\end{proof}
\begin{corollary}\label{cor:lqlpweakbound}
	Let $X$ be a bounded degree graph. For any $q \geq p \geq 1$ we have $\Lambda^q_X(r) \lesssim r^{1/p-1/q} \Lambda^p_X(r)$.
\end{corollary}
\begin{proof}
	Choose a subgraph $\Gamma$ with $\Lambda^q_X(r)=|\Gamma| h^q(\Gamma)$,
	and a non-constant function $f:V(\Gamma)\to\R$ with sum zero and $h^p(\Gamma) = \|\nabla f\|_p/ \|f\|_p$.

  By Lemma~\ref{lem:lplq-bounds},
  \begin{align*}
	\Lambda^q_X(r)
	& \leq |\Gamma |\, h^q(\Gamma)
	  \leq |\Gamma |\, \frac{\|\nabla f\|_q}{\|f\|_q}
	  \leq |\Gamma |\, r^{1/p-1/q} \frac{\|\nabla f\|_p}{\|f\|_p}
	\\ & = r^{1/p-1/q}\, |\Gamma|\, h^p(\Gamma) 
	\leq r^{1/p-1/q} \Lambda^p(X). \qedhere
  \end{align*}
\end{proof}

Since the $3$-regular tree $T_3$ has $\Lambda_{T_3}^p(r) \simeq r^{1-\frac{1}{p}}$~\cite[Theorem 10.1]{HumeMackTess-Pprof}, we observe the following.

\begin{corollary}\label{cor:crit-exponent-interval}
	If $X$ is a bounded degree graph, then
	\[
    \left\{ p \in [1,\infty) : \Lambda_X^p(r) \lesssim r^{1-\frac{1}{p}}  \right\}
	\]
	is an interval of the form $(\pcrit(X),\infty)$, $[\pcrit(X),\infty)$ or $\emptyset$ (with $\pcrit(X)=\infty$).

Moreover, if $X$ contains a regularly embedded $3$-regular tree, then
	\[
		\left\{ p \in [1,\infty) : \Lambda_X^p(r) \simeq r^{1-\frac{1}{p}}  \right\}
        = \left\{ p \in [1,\infty) : \Lambda_X^p(r) \lesssim r^{1-\frac{1}{p}}  \right\}.
	\]
\end{corollary}

There are examples of all three types of interval in the corollary: for $\HH^3$ this set is $(2,\infty)$, 
for the union of a tree and $\Z^2$ it is $[2,\infty)$,
and for $\HH^3\times\R$ the set is $\emptyset$.
Recall that if $X$ is a non-elementary hyperbolic group this set is non-empty since we have $\pcrit(X) \leq \Confdim(\bdry X) < \infty$, see Theorem~\ref{thm:crit-exponent-previous-upper-bound}.

\section{Round trees}
\label{sec:round-trees}

Gromov introduced ``round trees'' as negatively curved $2$-complexes admitting an isometric $S^1$ action with a single fixed point, so that there is an isometrically embedded tree which meets every fibre of the action at a single point~\cite[\S 7.C$_3$]{Gromov93}.
The second author used a more combinatorial version of this construction in the study of random groups~\cite{Mac-12-random-cdim, Mac-16-confdim-subcplxs}.

The first definition we have of a round tree is as follows,
\begin{definition}[{Combinatorial round tree \cite[Definition 7.1]{Mac-16-confdim-subcplxs}}]
	\label{def:comb-round-tree-complex}
	For $2 \leq H_{\min} \leq H_{\max}$, a \emph{combinatorial round tree} with vertical branching $V$ and horizontal branching $\in [H_{\min}, H_{\max}]$ is a polygonal $2$-complex $A$ where, setting $T = \{1,2,\ldots, V\}$, we can write
	\[ A = \bigcup_{\ba \in T^\N} A_\ba \]
	with the following properties.
	\begin{enumerate}
		\item $A$ has a base point $1$ contained in the boundary of a unique $2$-cell $A_\emptyset \subset A$.
		\item Each $A_\ba$ is an infinite planar $2$-complex homeomorphic to a half-plane whose boundary is the union of two rays $L_\ba$ and $R_\ba$ with $L_\ba \cap R_\bbb = L_\bbb \cap R_\ba = \{1\}$ for all $\ba,\bbb \in T^\N$.
		\item Set $A_0 = A_\emptyset$ and $A_n = \St(A_{n-1})$ for $n>0$.  
			For $\ba=(a_1,a_2,\ldots) \in T^\N$ let $\ba_n = (a_1,\ldots,a_n) \in T^n$.
			If $\ba,\bbb \in T^\N$ with $\ba_n=\bbb_n$ and $\ba_{n+1}\neq \bbb_{n+1}$ then
			\[ 
				A_n \cap A_\ba \subset A_\ba \cap A_\bbb \subset A_{n+1}\cap A_\ba.
			\]
		\item
			For each $\ba \in T^\N$ and $n \in \N$, going from $L_\ba$ to $R_\ba$, associate each $2$-cell in $A_\ba \cap (A_{n+1}\setminus A_{n})$ to the first $2$-cell $R \subset A_\ba \cap (A_n\setminus A_{n-1})$ with which it shares a $1$-cell, or if no such $R$ exists then the first with which it shares a $0$-cell.  Then each $R \subset A_\ba \cap (A_n \setminus A_{n-1})$ is associated with between $H_{\min}$ and $H_{\max}$ $2$-cells of $A_\ba \cap (A_{n+1}\setminus A_{n})$.
	\end{enumerate}
\end{definition}
Such complexes are Gromov hyperbolic.
One tricky feature of this definition is that (as is necessary in some applications), a $2$-cell in $A_n\setminus A_{n-1}$ can share $1$-cells with two different $2$-cells in $A_{n-1}$.  By adjusting the complex slightly to remove this ambiguity one gets a quasi-isometrically equivalent complex.  For this complex, if one takes the (quasi-isometrically equivalent) dual graph with a vertex for each $2$-cell in $A$ and an edge for $2$-cells in some $A_\ba$ which share a $1$-cell, one gets a ``round tree graph'' as below, which is easier to work with for our calculations.

For $\Delta \in \N$, let $[\Delta] = \{0,1,2,\ldots,\Delta-1\}$.
Recall $[\Delta]^*$ is the set of all finite words in the alphabet $[\Delta]$.
For $t \in [\Delta]^*$, we write $|t|$ for the length of a word $t$, and denote the concatenation of two words $t_0,t_1 \in [\Delta]^*$ by $t_0t_1$.
\begin{definition}
	\label{def:round-tree-graph}
	A \textbf{round tree graph} with vertical branching $V \geq 2$ and horizontal branching $\in [H_{\\min},H_{\max}]$ for $2 \leq H_{\min} \leq H=H_{\max}$ is a graph $\Gamma=(V(\Gamma),E(\Gamma))$ where the vertex set $V(\Gamma)$ is a subset of $[H]^* \times [V]^*$, and the following properties are satisfied. 
	\begin{enumerate}
		\item For all $t \in V$, $|\pi_H(t)|=|\pi_V(t)|$, where $\pi_H:V(\Gamma) \to [H]^*$ and $\pi_V:V(\Gamma) \to [V]^*$ denote the coordinate projections.
		\item For all $h \in [H]^*, v \in [V]^*$ and $\alpha \in [H], \beta \in [V]$, if $(h\alpha,v\beta) \in V$ then $(h,v) \in V$, and $(h,v)(h\alpha,v\beta)$ are connected by a \emph{vertical} edge in $E$.
		\item If $(h,v) \in V$ then for all $\beta \in [V]$ there exists $\Delta_{h,v,\beta} \in [H_{\\min},H_{\max}]$ so that
			\[
				\{\alpha \in [H] : (h\alpha, v\beta) \in V\} = [\Delta_{h,v,\beta}].
			\]
		\item If $(h,v),(h',v) \in V$ and $h,h'$ are consecutive in the lexicographic order on
			\[ \{ w \in [H]^* : (w,v) \in V \}, \]
			then $(h,v)(h',v)$ are connected by a \emph{horizontal} edge in $E$.
		\item Every edge in $E$ is either vertical or horizontal as above.
	\end{enumerate}
	If $H_{\min}=H_{\max}=H$ then $\Gamma$ is uniquely determined, and we call $\Gamma$ the \textbf{$(H,V)$-regular round tree graph}, and denote it by $\Gamma=RT^{H,V}$.
\end{definition}
For an example of a portion of such a graph, with a faint corresponding combinatorial round tree, see Figure~\ref{fig:rtgraph}.
\begin{figure}
   \def\svgwidth{.8\textwidth}
   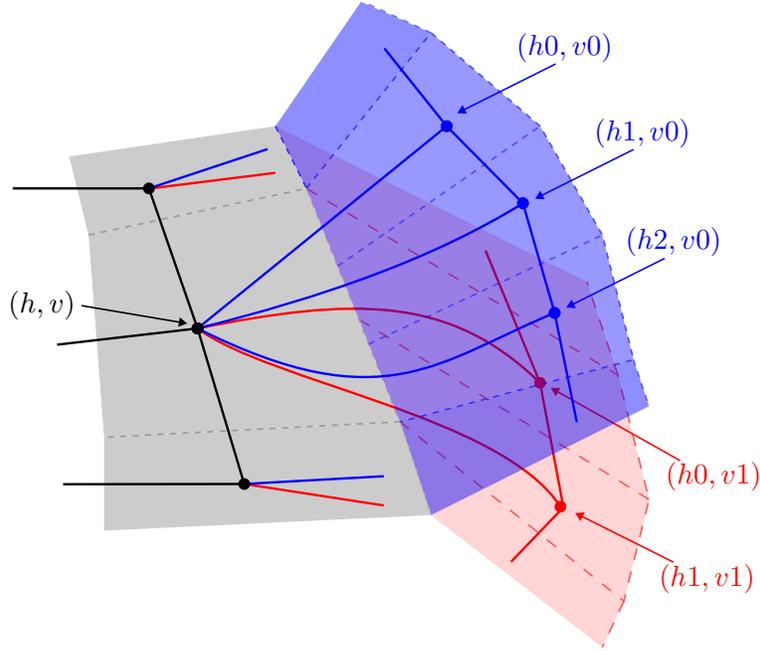
	\caption{Part of a round tree graph}	
	\label{fig:rtgraph}
\end{figure}

As before, such round tree graphs are Gromov hyperbolic, indeed are quasi-isometric to a $\CAT(-1)$ polygonal $2$-complex.  Finally, we recall known conformal dimension bounds for these round trees.
Given $H \geq 2, V \geq 1$, let $Z^{H,V}$ be the Cantor set $[V]^\N$ with the metric 
\[
	\rho_{H,V}((a_1,a_2,\ldots),(b_1,b_2,\ldots)) = \exp\left(-\log(H)\min\{i : a_i \neq b_i\}\right).
\]
This metric is Ahlfors regular of dimension $\log V / \log H$.
Indeed one can define an Ahlfors regular measure $\mu$ on $Z^{H,V}$ so that the ball of radius $H^{-k}$ around any point $(a_1,\ldots,a_k,\ldots)$ has measure $V^{-k}$.
\begin{theorem}
	\label{thm:round-tree-confdim}
	For any round tree graph $\Gamma$ with vertical branching $V$ and horizontal branching in $[2,H]$, one has
	\[
		\Confdim(\bdry \Gamma) \geq 1+\frac{\log V}{\log H}.
	\]
	Moreover, $\bdry RT^{H,V}$ carries a visual metric bi-Lipschitz to $Z^{H,V} \times[0,1]$, and so $\Confdim(\bdry RT^{H,V})= 1+\frac{\log V}{\log H}$ is attained.
\end{theorem}
\begin{proof}
	We have that $\Gamma$ is the dual graph to a combinatorial round tree with vertical branching $V$ and horizontal branching in $[2,H]$ in the sense of Definition~\ref{def:comb-round-tree-complex}.  The lower bound then follows from \cite[\S 6]{Mac-12-random-cdim}, see \cite[Theorem 7.2]{Mac-16-confdim-subcplxs}.

	For the upper bound: $RT^{H,V}$ is the dual graph to a $\CAT(-1)$ complex defined as follows.  Let $o \in \HH^2$ be the origin, and choose a half-plane $U$ with $o \in \partial U$.  Divide $U$ into concentric annuli $U_i = B(o,(i+1)\log(H))\setminus B(o,i\log(H))$, $i \geq 0$, of width $\log(H)$.
	Divide each $U_i$ into $H^i$ $2$-cells using equally spaces radii centred at $o$.  Then $RT^{H,1}$ is exactly the dual graph of $U$ with this $2$-complex structure.  Moreover $RT^{H,V}$ is the dual graph to the $\CAT(-1)$ complex $W$ resulting from gluing copies $\{U_v\}_{v \in [V]^\N}$ of $U$ together along the convex subsets $B(o,i\log(H))$ in the appropriate way.  Since the circumference of $\partial B(o,(i+1)\log(H)) \asymp H^i$ one sees that each $U$ is quasi-isometric to $RT^{H,1}$, and $W$ is quasi-isometric to $RT^{H,V}$.

	Being $\CAT(-1)$, $W$ carries a visual metric $\rho(a,b) \asymp e^{-(a|b)_o}$; on each $\bdry U_i$ this restricts to a metric (uniformly) bi-Lipschitz to $[0,1]$.
	Define a Lipschitz projection $\pi:W \to U$ by isometrically sending each $U_i \to U$.  Fix a visual metric $\asymp e^{-(a|b)_o}$ on $U$ and a bi-Lipschitz parametrization $h:\bdry U \to [0,1]$.
	For $a \in \bdry W$, $a$ is the limit point of a geodesic ray $\gamma_a$ lying entirely in a unique $U_{v(a)} \subset W$, $v(a)\in[V]^\N$.
	Define $f: \bdry W \to Z^{H,V} \times [0,1]$ by $f(a)=(v(a), \bdry (\pi \circ \gamma_a))$; metrize $Z^{H,V}\times[0,1]$ by $d'((v,t),(v',t')) = \max\{\rho_{H,V}(v,v'),|t-t'|\}$.
	
	We claim $f$ is bi-Lipschitz.
	Suppose $a, b \in \bdry W$.  
	Let $v(a)=(a_1,\ldots), v(b)=(b_1,\ldots)$.
	Let $k = \min \{ i : a_i \neq b_i\}$, so $\rho_{H,V}(v(a),v(b)) = H^{-k}$.
	Let $r = (\pi\circ \gamma_a | \pi\circ \gamma_b)_o$, so $|h(\bdry(\pi\circ\gamma_a))-h(\bdry(\pi\circ\gamma_b))| \asymp e^{-r}$.
	If $r \leq k \log H$ then the point at which the geodesic rays $\pi\circ \gamma_a, \pi\circ\gamma_b$ coarsely diverge lifts to a point in $U_{v(a)}\cap U_{v(b)}$, thus
	\[
		\rho(a,b) \asymp e^{-r} \asymp \max\{e^{-r},H^{-k}\} = d'(f(a),f(b)).
	\]
	If $r > k \log H$ then the geodesics in $W$ to $a,b$ coarse diverge at the point $k \log H$ along either $\gamma_a$ or $\gamma_b$, and so
	\[
		\rho(a,b) \asymp e^{-k\log H} \asymp \max\{e^{-r},H^{-k}\} = d'(f(a),f(b)).
	\]
	So in either case, we see that $f$ is bi-Lipschitz.

	As $Z^{H,V}\times [0,1]$ is Ahlfors regular with dimension $1+\frac{\log V}{\log H}$, the attainment of this bound for the conformal dimension of $\bdry RT^{H,V}$ is standard~\cite[Example 4.1.9]{Mac-Tys-cdimexpo}.
\end{proof}

\section{Poincar\'e constants of round trees}
\label{sec:reg-round-trees}

The goal of this section is to calculate the critical exponent of Poincar\'e profiles for round tree graphs, and to bound their separation profiles.

\begin{theorem}\label{thm:pcritRT}
	If $\Gamma$ is a round tree graph with horizontal branching in $[2,H]$ and vertical branching $V \geq 2$, then:
	\[
		\pcrit(\Gamma) \geq 1 + \frac{\log V}{\log H}.
	\]
	In particular, if $\Gamma = RT^{H,V}$ is the $(H,V)$-regular round tree graph,
	\[
    	\pcrit(RT^{H,V})=1+ \frac{\log V}{\log H} =\Confdim(\bdry RT^{H,V}).
    \]
\end{theorem}

This follows from the following proposition, the main result of this section.
\begin{proposition}
	\label{prop:roundtree-lower}
	Let $\Gamma$ be a round tree graph with horizontal branching in $[2,H]$ and vertical branching $V \geq 2$.  Let $Q= 1+\frac{\log V}{\log H}$.
    Then for all $1\leq p < Q$, there exists $c>0$ so that, setting \[ \epsilon = \frac{Q-p}{p(Q-p+pQ)} > 0, \]
	for any $r >1$ there exists a subgraph $Y \subset \Gamma$ of size $|Y| \asymp r$ and with
	\[
		\Lambda_\Gamma^p(|Y|) \geq |Y| h^p(Y) \geq c |Y|^{1-\frac{1}{p}+\epsilon}.
	\]
	Thus $\pcrit(\Gamma)\geq Q$.
\end{proposition}
Note that this proposition implies that for all $1\leq p < Q$,
\[
    \Lambda_\Gamma^p(r) \succeq r^{1-\frac{1}{p+ \frac{Q-p}{Q}}}= r^{\frac{pQ-p}{pQ-p+Q}}
\]
which, in the case $p=1$, simplifies to $\Lambda_\Gamma^1(r) \succeq r^\frac{Q-1}{2Q-1}$. 

In the proof of Proposition \ref{prop:roundtree-lower} we require the following standard lemma.
\begin{lemma}[{e.g.\ \cite[Proof of Theorem 10.1]{HumeMackTess-Pprof}}]
	\label{lem:pathcounting}
		Suppose for any vertices $w_0,w_1$ in a finite graph $\Gamma$ we have a choice of path $\gamma_{w_0,w_1}$ from $w_0$ to $w_1$. For each edge $e$, let $M_e$ denote the number of pairs $(w_0,w_1) \in V(\Gamma)^2$ with $e \subset \gamma_{w_0,w_1}$.

		Then for $p>1$ we have:
\[
 h^p(\Gamma) \geq  \frac{|\Gamma|^{\frac1p}}{\max_{w_0,w_1}\left(\sum_{e\in\gamma_{w_0,w_1}} M_e^{\frac{1}{p-1}}\right)^{\frac{p-1}{p}}},
\]
		and for $p=1$ we have:
		\[
			h^1(\Gamma) \geq \frac{|\Gamma|}{\max_{e \in E(\Gamma)} M_e}.
		\]
	\end{lemma}
	\begin{proof}
		As in \cite[Proposition 10.3 and Proof of Theorem 10.1]{HumeMackTess-Pprof}, for any $f:V(\Gamma)\to \R$ and $p > 1$, letting $f_\Gamma = \frac{1}{|\Gamma|}\sum_{x \in V(\Gamma)}f(x)$,
		\begin{align*}
			\|f-f_\Gamma\|_p^p 
 			& \leq \frac{1}{|\Gamma|} \sum_{w_0,w_1 \in V(\Gamma)} |f(w_0)-f(w_1)|^p
			\\ & \leq \frac{1}{|\Gamma|} \Bigg( \sum_{\substack{w_0,w_1 \in V(\Gamma)\\ e \subset \gamma_{w_0,w_1}}} |\nabla f(e)|^p M_e^{-1}\Bigg) \max_{w_0,w_1 \in V(\Gamma)} \Bigg(\sum_{e \subset \gamma_{w_0,w_1}} M_e^{\frac{1}{p-1}}\Bigg)^{p-1}
			\\ & \leq \frac{1}{|\Gamma|} \|\nabla f\|_p^p \max_{w_0,w_1 \in V(\Gamma)} \Bigg(\sum_{e \subset \gamma_{w_0,w_1}} M_e^{1/(p-1)}\Bigg)^{p-1}.
		\end{align*}
		The same proof simplifies to show
		\[
			\|f-f_\Gamma\|_1 \leq \frac{1}{|\Gamma|} \|\nabla f\|_1 \max_{e \in E(\Gamma)} M_e;
		\]
		compare e.g.\ \cite[Proposition 1]{GladShum}.
	\end{proof}

\begin{proof}[Proof of Proposition~\ref{prop:roundtree-lower}]
	Fix $p<Q$, and $r>0$.  

	Let $T \in \N, k \in \N$ be constants to be determined.
	
	In $\Gamma$ fix a horizonal path $Y_0$ of length $T$, say with initial coordinate $(0^t,0^t)$ for some large enough $t$.  In particular, $\pi_V|_{Y_0}$ has image a single point with preimage a horizontal path of length $T$.
	We build the subcomplex $Y$ by repeating the following process, for $i=1, \ldots, k$.
	\begin{enumerate}
		\item For each maximal word $v \in \pi_V(Y_{i-1})$, assume $\pi_V^{-1}(v) \subset Y_{i-1}$ is a horizontal path of length $\in [TH^{i-1},TH^i)$.
		\item For each $\beta \in [V]$, add to $Y_{i}$ all points $(hu,v\beta0^j) \in \Gamma$ so that $(h,v) \in \pi_V^{-1}(v)$, with $j$ minimal (depending on $v$) so that the horizontal path $\pi_V^{-1}(v\beta0^j)$ has length $\geq TH^i$.
	\end{enumerate}
	Observe at step 2 that each time the vertical coordinate adds another letter, the horizontal path length increases by a factor between $2$ and $H$.  So the $j$ in step 2 satisfies $0 \leq j \leq \lceil \log_2(H)-1 \rceil$.

	Let $\Gamma_V = \pi_V(Y_k)$ and observe that $\Gamma_V$ is a rooted tree with $V^k$ leaves, with all vertices having either $1$ or $V$ children, and with the property that any path of vertices each having only $1$ child has length at most $\lceil \log_2(H) -2 \rceil$.

	By construction, $|Y_k| \asymp T(HV)^k \asymp TH^{Qk}$.

	Define a colouring map $\Pi:Y_k \times Y_k \to [T]$ with the following property:
	for each fixed $x \in Y_k$, and each fixed $t \in [T]$, 
	\[
		\big| \Pi^{-1}(t) \cap \{x\} \times Y_k \big| \preceq (HV)^k \asymp H^{Qk}.
	\]
	Likewise, 
	for each fixed $y \in Y_k$, and each fixed $t \in [T]$, 
	\[
		\big| \Pi^{-1}(t) \cap Y_k \times \{y\} \big| \preceq (HV)^k \asymp H^{Qk}.
	\]
	Such a map $\Pi$ can be chosen by (roughly) labelling each vertex in $Y_k$ by labels in $[T^{1/2}]$ alternating along horizontal paths, then letting $\Pi(x,y)$ be the combination of labels in $[T^{1/2}]\times[T^{1/2}] \approx [T]$.

	Now to bound the Poincar\'e constants in $Y_k$, we choose paths as in Lemma~\ref{lem:pathcounting}:

	Let $(b_t)_{t \in [T]}$ label the horizontal path $\pi_V^{-1}(Y_0)$.
	For $(w_0,w_1) \in Y_k$, with $w_0=(h_0,v_0), w_1=(h_1,v_1)$, choose
	$w_0'=(h_0',v_0) \in Y_k$ with $b_{\Pi(w_0,w_1)}$ an initial subword of $h_0'$, 
	and choose $w_1'=(h_1',v_1) \in Y_k$ with $b_{\Pi(w_0,w_1)}$ an initial subword of $h_1'$.
	Now define $\gamma_{w_0,w_1}$ by travelling horizontally from $w_0$ to $w_0'$, 
	then travel vertically to $(b_{\Pi(w_0,w_1)},0^t) \in Y_0$, 
	then travel vertically out to $w_1'$, then travel horizontally to $w_1$.

\smallskip
{\noindent\textbf{Upper bounds on $M_e$:}}

    For a horizonal edge $e = ((h,v),(h',v))$, $e \in \gamma_{w_0,w_1}$ implies that either $\pi_V(w_0)=v$ (or $\pi_V(w_1)=v$), leaving $\preceq TH^k$ choices for $w_0$ (or $w_1$) and $\asymp TH^{Qk}$ choices for $w_1$ (or $w_0$). Hence,
    \begin{equation}\label{eq:horizedge}
        M_e \preceq TH^{k}TH^{Qk}. 
    \end{equation}

    For a vertical edge $e = ((h,v),(h',v'))$ with $|v|>|v'|$ (or $|v|<|v'|$) we know $b_{\Pi(w_0,w_1)}$ so there are $\preceq TH^{Qk} H^{Qk} \asymp |Y_k|^2/T$ possibilities for $(w_0,w_1)$. Additionally, for each branch vertex in $\Gamma_V$ between $v$ and $\pi_V|_{Y_0}$, the number of possibilities for $w_0$ (or $w_1$) goes down by a factor of $1/V$, so we have
    \begin{equation}\label{eq:vertedge}
        M_e \preceq \frac{|Y_k|^2}{T} \cdot V^{- \left\lfloor\frac{|v|-|\pi_V|_{Y_0}|}{\log_2H}\right\rfloor}.
    \end{equation}

    We now apply Lemma \ref{lem:pathcounting}. 
    
\smallskip\noindent\textbf{Case $p>1$:}

	We wish to bound $\sum_{e \in \gamma_{w_0,w_1}} M^{1/(p-1)}$ from above independently of $w_0,w_1 \in Y_k$.
	
	By \eqref{eq:horizedge}, the sum along the horizontal edges in $\gamma_{w_0,w_1}$ satisfies
	\begin{equation}
		\label{eq:paths1}
		\sum M_e^{\frac{1}{p-1}} \preceq (TH^k) \left( TH^k |Y_k| \right)^{\frac{1}{p-1}} = (TH^k)^{\frac{p}{p-1}} |Y_k|^{\frac{1}{p-1}}.
	\end{equation}

	Now, using \eqref{eq:vertedge}, the values $M_e^{\frac{1}{p-1}}$ along the edges of a vertical path can be bounded by a sum of two geometric series, each of which has maximum term $\asymp \left(|Y_k|^2 /T \right)^{1/(p-1)}$ and ratio $V^{-1/\log_2(H)}$. Therefore,
	\begin{equation}
		\label{eq:paths2}
		\sum M_e^{\frac{1}{p-1}} \preceq \left( |Y_k|^2 / T  \right)^{\frac{1}{p-1}}.
	\end{equation}
	
	Put together, by Lemma~\ref{lem:pathcounting} we have
	\begin{equation*}
		|Y_k| h^{p}(Y_k)  \succeq \frac{|Y_k|^{1+\frac{1}{p}}}{ \max\{ TH^k |Y_k|^{1/p}, (|Y_k|^2/T)^{1/p} \} }.
	\end{equation*}

	Now we fix variables. Choose $k\in\N$ and set $T$ to satisfy $T^p\asymp H^{k(Q-p)}$, so $|Y_k|\asymp T^{1+\alpha Q}$ with $\alpha=\frac{p}{Q-p}$.  Note that by varying $k$ we can ensure $|Y_k| \asymp r$ for any $r\geq 1$.  The above equation may be rewritten as follows:
    \begin{align*}
		|Y_k| h^{p}(Y_k)   & \succeq \frac{T^{(1+\alpha Q)(1+\frac1p)}}{\max\{T^{1+\alpha}T^{(1+\alpha Q)/p},T^{(1+2\alpha Q)/p}\}}  \\ & = \min\{T^{\alpha(Q-1)},T^{1+\alpha Q-\alpha Q/p}\}.
	\end{align*}

    It suffices to prove that there is some $\epsilon>0$ such that 
    \[
        \min\{T^{\alpha(Q-1)},T^{1+\alpha Q-\alpha Q/p}\} \succeq T^{(1+\alpha Q)(1-\frac1p+\epsilon)}
    \]
    or, in other words, that
    \[
        \min\{\alpha(Q-1),1+\alpha Q-\alpha Q/p\} \geq (1+\alpha Q)(1-\frac1p+\epsilon).
    \]
    Rearranging $\alpha(Q-1)= (1+\alpha Q)(1-\frac1p+\epsilon)$, we obtain
    \[
     \epsilon = \frac{1}{p}-\frac{1+\alpha}{1+\alpha Q}.
    \]
    Recall that $\alpha=\frac{p}{Q-p}$, so the above equation gives
    \[
    \epsilon = \frac{1}{p} - \frac{1+\frac{p}{Q-p}}{1+\frac{pQ}{Q-p}} = \frac{Q-p}{p(Q-p+pQ)}>0.
    \]
    Similarly, rearranging $1+\alpha Q-\alpha Q/p = (1+\alpha Q)(1-\frac1p+\epsilon)$, we have
    \[
    \epsilon = \frac{1}{p(1+\alpha Q)}.
    \]
    Again, setting $\alpha=\frac{p}{Q-p}$ in the above equation we get
    \[
    \epsilon = \frac{1}{p\left(1+\frac{pQ}{Q-p}
    \right)} = \frac{Q-p}{p(Q-p+pQ)}>0.
    \]
    Thus, Proposition \ref{prop:roundtree-lower} holds in the case $p>1$ with $\epsilon= \frac{Q-p}{p(Q-p+pQ)}$.

\smallskip\noindent\textbf{Case $p=1$:}

	As in the $p>1$ case, choose $k\in\N$ and set $T$ to satisfy $T \asymp H^{k(Q-1)}$, so $|Y_k|\asymp T^{1+\alpha Q}$ with $\alpha=\frac1{Q-1}$. Now, Lemma \ref{lem:pathcounting} gives
	\[
		|Y_k|h^1(Y_k) \succeq \frac{|Y_k|^2 }{\max\{TH^k |Y_k|, |Y_k|^2/T\}}
		\asymp \min\{T^{\alpha(Q-1)},T\} = T
	\]
	It suffices to find $\epsilon>0$ such that 
    \[
        T\succeq T^{(1+\alpha Q)\epsilon}\asymp |Y_k|^\epsilon \asymp |Y_k|^{1-\frac1p+\epsilon}
    \]
    or, in other words, $\epsilon(1+\alpha Q)\leq 1$. It is clear that this holds for 
    \[
		\epsilon= \frac1{1+\alpha Q} = \frac{Q-1}{2Q-1} = \frac12 - \frac{1}{2(2Q-1)}>0.\qedhere
    \]
\end{proof}

\section{Applications of regular round trees}
\label{sec:applications-reg-round-trees}
In this section we use the results above to compute or bound the critical exponent $\pcrit$ for certain hyperbolic cones on product spaces, Heintze manifolds, and random groups.

\subsection{Product spaces}
\begin{theorem}
	\label{thm:product-crit-exponent}
	Suppose $Z$ is a compact Ahlfors $Q$-regular space for some $Q>0$.
	Then for all $1 \leq p<1+Q$ there exists $\epsilon>0$ so that
	$\Lambda_{\Cone(Z \times [0,1])}^p(r) \gtrsim r^{1+\frac{1}{p}+\epsilon}$,
	in particular:
	\[
		\Confdim(Z \times [0,1]) = \pcrit(\Cone(Z \times [0,1])) = 1+Q.
	\]
\end{theorem}
In the following proof, recall that the spaces $Z^{H,V}$ were defined in Section~\ref{sec:round-trees}.
\begin{proof}
	That $\Confdim(Z \times [0,1])=1+Q$ essentially goes back to Pansu~\cite[Example 3.4]{Pan-89-cdim}; for discussion see \cite[Example 4.1.9]{Mac-Tys-cdimexpo}.

	For any $p \in [1,1+Q)$, choose $H \geq 2, n \geq 2$ so that for $V=2^n$, we have $p \leq 1+\frac{\log V}{\log H} < 1+Q$.
	Then by \cite[Theorem 3.1]{Mattila-Saaranen-09-existence-AR-cantor-subsets} there exists a bi-Lipschitz embedding $Z^{H,V} \to Z$.
	Thus there exists a bi-Lipschitz embedding $Z^{H,Z}\times[0,1] \to Z \times [0,1]$, which extends to a quasi-isometric embedding $\Cone(Z^{H,Z} \times [0,1]) \to \Cone(Z \times [0,1])$.
	If $p \geq \frac{1}{2}(1+\frac{\log V}{\log H})$ then we are done by Proposition~\ref{prop:roundtree-lower}.
	Any remaining case where $p < \frac{1}{2}(1+\frac{\log V}{\log H})$ then follows from Corollary~\ref{cor:lqlpweakbound}.
\end{proof}

\subsection{Heintze manifolds $\R^n\rtimes_{a_1,\ldots,a_n}\R$}

The bounds above apply to the following groups.
\begin{definition} Given $\underline{a}=(a_1,\ldots,a_n)\in\R^n$, we denote by $\R^n\rtimes_{\underline{a}}\R$, the semidirect product where the action of $\R$ is given by
\[
 \psi(t)(x_1,\ldots,x_n) = (\exp(a_1t)x_1,\ldots,\exp(a_nt)x_n).
\]
\end{definition}
Note that for non-zero $\lambda\in\R$, $\R^n\rtimes_{\underline{a}}\R$ and $\R^n\rtimes_{\lambda \underline{a}}\R$ are isomorphic.

When all $a_i > 0$, the group $\R^n\rtimes_{\underline{a}}\R$ is Gromov hyperbolic, indeed, it is a Heintze group admitting a metric of negative sectional curvature.
Assuming $a_1=\min a_i$, the boundary $\bdry (\R^n\rtimes_{\underline{a}}\R)$, minus a point, is quasi-M\"obius equivalent to
\[
	(\R,d^{1/a_1}) \times(\R,d^{1/a_2}) \times \cdots \times (\R,d^{1/a_n})
\]
and hence to 
\[
	(\R,d) \times(\R,d^{a_1/a_2}) \times (\R,d^{a_1/a_3}) \times \cdots \times (\R,d^{a_1/a_n})
\]
where the metric $d$ denotes the standard metric on $\R$; see for example \cite[Section 2.5]{Dymarz-10-diagonal-Tukia}.  Having this local product structure, the conformal dimension is well-known.

\begin{theorem}[{\cite[Example 3.4]{Pan-89-cdim}}]  When $0 < a_1\leq a_2 \leq \ldots \leq a_n$, $\R^n\rtimes_{\underline{a}}\R$ is hyperbolic and the conformal dimension of its boundary is $\sum_{i=1}^n \frac{a_i}{a_1}$.
\end{theorem}

We now see that the conformal dimension of the boundary equals the critical exponent $\pcrit$.
\begin{corollary}\label{cor:crit-exponent-heintze}
	Let $G=\R^n\rtimes_{\underline{a}}\R$. For every $1\leq p < \Confdim(\bdry G)$ there is some $\varepsilon=\varepsilon(p)>0$ such that
\[
 \Lambda^p_G(r)\gtrsim r^{1-1/p+\varepsilon}.
\]
	In particular, $\pcrit(G) = \Confdim(\bdry G)$.
\end{corollary}
\begin{proof}
	Follows from Theorem~\ref{thm:product-crit-exponent}, as $\bdry G$ contains a bi-Lipschitz copy of
	\[
		[0,1] \times \left( ([0,1],d^{a_2/a_1}) \times \cdots \times ([0,1], d^{a_n/a_1})\right). \qedhere
	\]
\end{proof}

\begin{corollary} If there is a regular map $\R^n\rtimes_{\underline{a}}\R\to\R^n\rtimes_{\underline{a'}}\R$, where $1= a_1\leq \ldots \leq a_n$ and $1= a'_1\leq \ldots \leq a'_{n'}$, then $\sum_{i=1}^n a_i \leq \sum_{j=1}^{n'} a'_j$.
\end{corollary}

\subsection{Random groups}
\label{ssec:random-groups}

We now apply the round tree bounds above to certain random groups.   

We consider random groups in the sense of Gromov: 
fix an integer $m \geq 2$, a function $n=n(\ell):\N\to\N$ and a property $P$ of a group presentation.
Given $\ell\in \N$, let $G= \langle s_1,\ldots, s_m | r_1, \ldots, r_n\rangle$
be a group presentation where each $r_i$ is chosen independently and uniformly at random from all cyclically reduced words of length $\ell$ in $\langle s_1,\ldots, s_m \rangle$.  
If $P$ holds with probability $\to 1$ as $\ell \to \infty$, we say $P$ holds \emph{asymptotically almost surely (a.a.s.)}.

When $n = \lfloor C\ell^K \rfloor$ for some constants $C>0$ and $K \geq 0$ this is the \emph{polynomially many relator model (of degree $K$)}. The degree $0$ case, When $n$ is constant, is the random \emph{few (or $n$-)relator model}. When $n = \lfloor (2m-1)^{\ell d} \rfloor$ for some constant \emph{density} $d\in(0,1)$ this is \emph{(Gromov's) density model}.

In all these models, a.a.s.\ the random group is hyperbolic.  In the few and polynomially many relator models with $n \geq 1$, and in the density model with $d \in (0,1/2)$, a.a.s.\ the random group is non-elementary hyperbolic with Menger sponge boundary.  See \cite{Oll-05-rand-grp-survey, Mac-16-confdim-subcplxs} for further details.

Combinatorial round trees have been embedded into random groups in these models,
see 
\cite[\S 7.2]{Mac-16-confdim-subcplxs}, \cite[\S 8]{Mac-16-confdim-subcplxs} and \cite[\S 3.2, \S 5]{Frost-22-confdim-random}.
We thus have:
\begin{corollary}\label{cor:crit-exponent-rand-grp}
	A polynomial degree $K$ random group $G$ satisfies 
	\[
		2+K-\frac{5\log\log\ell}{\log\ell} 
		\leq \pcrit(G) 
		\leq \Confdim(\bdry G)
		\leq 2+K + \frac{2(K+1)\log\log\ell}{\log\ell}
	\] 
	a.a.s.\ as $\ell \to \infty$.

	A density $d \in (0,\frac{1}{2})$ random group satisfies 
	\[
		\frac{d(1-2d)^5 \ell \log(2m-1)}{C|\log(d(1/2-d))|} 
		\leq \pcrit(G) 
		\leq \Confdim(\bdry G)
		\leq \frac{Cd\ell \log(2m-1)}{(1-2d)|\log d|}
	\]
	a.a.s.\ as $\ell \to \infty$.
\end{corollary}
In particular, for $K'>K$, polynomial degree $K'$ random groups do not coarsely embed in polynomial degree $K$ groups,
and at each density $d\in (0,\frac{1}{2})$, high $\ell$ random groups do not coarsely embed in low $\ell$ random groups.

\begin{remark}
	Note that there is a sharper lower bound on the conformal dimension for random groups in Gromov's at density $d$ close to $\frac{1}{2}$ by Oppenheim~\cite{Oppenheim-23-FLp-rand-gromov-model}, see also discussion in~\cite[Section 1]{Frost-22-confdim-random}.  This uses spectral methods as opposed to round trees so there is no immediate equivalent lower bound on $p_\Lambda$.
\end{remark}

\subsection{(Relatively) hyperbolic Coxeter groups}
\label{ssec:coxeter-application}

Very recently, Field--Gupta--Lyman--Stark~\cite{Field-Gupta-Lyman-Stark-25-coxeter-confdim} and Cashen--Dani--Schreve--Stark~\cite{Cashen-Dani-Schreve-Stark-25-RACG-confdim} have built quasi-isometrically embedded round trees in certain (relatively) hyperbolic Coxeter groups, getting lower bounds on the conformal dimension of the boundary as a result, and infinitely many quasi-isometry classes of groups. 
By Theorem~\ref{thm:pcritRT} we have the same lower bounds on $p_\Lambda$, with consequent non-embedding results.  For example:

\begin{corollary}
	[{cf.\ \cite[Theorem A]{Field-Gupta-Lyman-Stark-25-coxeter-confdim}}]
	Fix $M \geq 3, m \geq 11$.  Let $W_\Gamma$ be a Coxeter group whose underlying graph $\Gamma$ has $m$ vertices and all edge labels in $[3,M]$.  Then
	\[
		p_\Lambda(W_\Gamma) \geq 1+ \frac{\log(\lfloor \frac{m-5}{3}\rfloor)}{\log(2M-1}.
	\]
\end{corollary}
Such Coxeter groups need not be hyperbolic, but are hyperbolic relative to two-dimensional flats.  In this case $p_\Lambda(W_\Gamma)$ is bounded from above by the conformal dimension of the Bowditch space of $W_\Gamma$, as in \cite[Theorem B]{Field-Gupta-Lyman-Stark-25-coxeter-confdim}.

In another direction:
\begin{corollary}
	There is a sequence $(G_i)_{i \in \N}$ of hyperbolic Coxeter groups with Pontryagin sphere boundary with $p_\Lambda(G_i)$ increasing to $\infty$ with $i$; in particular, for any $i<j$, $G_j$ does not regularly map into $G_i$
	(cf.\ \cite[Theorem D]{Field-Gupta-Lyman-Stark-25-coxeter-confdim}).
	These hyperbolic Coxeter groups may be taken to be right-angled~\cite[Theorem 5.1]{Cashen-Dani-Schreve-Stark-25-RACG-confdim}.	
\end{corollary}

And finally:
\begin{corollary}
	[{cf.\ \cite[Theorem 3.1]{Cashen-Dani-Schreve-Stark-25-RACG-confdim}}]
	For every $n\geq 2$ there exist hyperbolic right-angled Coxeter groups $(W_i)_{i\in \N}$ that virtually algebraically fibre, have virtual cohomological dimension $n$, and have $p_\Lambda(W_i)$ increasing to $\infty$ with $i$; in particular, for any $i<j$, $W_j$ does not regularly map into $W_i$.
\end{corollary}

\section{Poincar\'e profiles detect hyperbolic groups of conformal dimension $1$}
\label{sec:no-local-cut-points}

In this section we show Theorem~\ref{thm:power-sep-hyp-no-local-cut-points}: that if an infinite hyperbolic group $G$ does not split over a finite or virtually cyclic group, either $G$ is virtually Fuchsian, or $\sep_G(r) \simeq \Lambda_G^1(r) \gtrsim r^\alpha$ for some $\alpha>0$.

The idea is that the construction of \cite{Mac-10-confdim}, which builds a ``Cantor set of quasi-arcs'' in the boundary of such hyperbolic groups, has as its convex hull a subset of the Cayley graph which is quasi-isometric to a round tree graph in the sense of Definition \ref{def:round-tree-graph}.  The result then follows from Proposition~\ref{prop:roundtree-lower}.

\begin{proof}[Proof of Theorem~\ref{thm:power-sep-hyp-no-local-cut-points}]
	Choose a finite symmetric generating set of $G$ and let $d$ be the word metric on $G$ with respect to this generating set.
	Fix a visual metric $\rho$ on $\bdry G$ with $\rho(z,z') \asymp_C e^{-\epsilon (z|z')_1}$, normalized to have diameter $1$. We recall that, up to bounded additive errors, the geodesic triangle with endpoints $1,z,z'$ is a tripod where the Gromov product $(z|z)'_1$ is the distance from $1$ to the centre of the tripod.
    
	Since $G$ doesn't split over a finite group the boundary $\bdry G$ is connected.
	By Bowditch, either $G$ is virtually Fuchsian or has no local cut points, so assume the latter.
	As in \cite[Proof of Corollary 1.2]{Mac-10-confdim}, $\bdry G$ is then annularly linearly connected, doubling and complete (see op.\ cit.\ for definitions), so \cite[Theorem 1.4]{Mac-10-confdim} holds.  

	We use the following notation for arcs, that is, (sub)spaces of metric spaces homeomorphic to $[0,1]$.
	For points $x,y$ in an arc $A$, $A[x,y]$ denotes the unique subarc with endpoints $x$ and $y$; likewise $A[x,y)$ denotes $A[x,y] \setminus \{y\}$, etc.
	An arc $A$ is a \emph{$\iota$-local $\lambda$-quasi-arc} if for all $x,y \in A$ with $d(x,y) \leq \iota$ then $\diam A[x,y] \leq \lambda d(x,y)$.
	For arcs $A,B$ in a metric space and $\iota>0$, $B$ \emph{$\iota$-follows} $A$ if there exists a not necessarily continuous map $p:B \to A$ such that for any $x,y \in B$, $B[x,y]$ is contained in the $\iota$-neighbourhood of $A[p(x),p(y)]$.

	By \cite[Proof of Theorem 1.4]{Mac-10-confdim}, we can construct arcs  in $\bdry G$ as follows.
	There exists $\lambda \geq 1$ so that 
	for any $\beta< 1/32$ small enough ($\beta =$`$\alpha\delta^*/32\lambda$' in the cited proof, but $\delta^*$ can be shrunk arbitrarily and the proof goes through verbatim; we only additionally require $\beta < 1/12\lambda^2$ below),
	there exist 
	arcs $J_\ba$ for any finite or infinite sequence $\ba = (a_1,a_2,\ldots,a_n) \in \{0,1\}^n$ or $\ba= (a_1,a_2,\ldots) \in \{0,1\}^\N$, with the following properties.
	(We denote the empty sequence by $\emptyset$.)
	Each $J_{\ba}$ has diameter at least $1/2$, 
	each $J_{a_1\cdots a_n}$ is a $\beta^n$-local $\lambda$-quasi-arc which $\beta^n$-follows $J_{a_1\cdots a_{n-1}}$; 
	furthermore, we may assume the `following' maps map endpoints to endpoints.
	Each $J_{a_1\cdots}$ has Hausdorff distance $\leq 2\lambda\beta^n$ to each $J_{a_1\cdots a_n}$ by \cite[Lemma 4.1]{Mac-10-confdim}.
	Moreover, for all $\ba=(a_1,\ldots), \bbb=(b_1,\ldots) \in \{0,1\}^\N$, both the minimum and the Hausdorff distance between $J_\ba$ and $J_\bbb$ are $\asymp \beta^n$ where $n = \min\{k: a_k\neq b_k\}$.
	Finally, there exists a uniform $\lambda'$ so that each $J_{\ba}, \ba \in \{0,1\}^* \cup \{0,1\}^\N$ is a (not local) $\lambda'$-quasi-arc, i.e.\ for all $x,y\in J_{\ba}$, $\diam(J_{\ba}[x,y])\leq \lambda'd(x,y)$.

    Let $W \subset G$ be the union of geodesic rays from $1 \in G$ to points in $J_\ba$ for $\ba \in \{0,1\}^\N$. By hyperbolicity, $W$ is a quasi-convex subset of $G$. Theorem \ref{thm:power-sep-hyp-no-local-cut-points} now reduces to the following proposition.
	\begin{proposition}\label{prop:QI-round-tree-graph} There is an (irregular) round tree graph $A$ which is quasi-isometric to $W$.
    \end{proposition}

    We will prove this proposition via a series of Lemmas.
    
	For each $n$, and each $\ba \in \{0,1\}^n$, take a maximal $\beta^n / 4\lambda$-separated net in $J_{\ba}$ which includes its endpoints and use the arc order to order it as a sequence $(z_{\ba}^{i})_{i=1}^{N_\ba}$.
	For a constant $C_1$ to be determined shortly, define each corresponding point $x_{\ba}^i$ in $G$ by taking a point at distance $\frac{-1}{\epsilon}\log(\beta^n)+C_1$ from $1$ along a geodesic ray to $z_{\ba}^i$.  For convenience set $x_\emptyset^1 = 1 \in G$.  
	By hyperbolicity, we can choose and fix $C_1=C_1(\lambda,C,\epsilon,\delta)$ large enough so that all points in $\{x_{\ba}^i\}_{i=1}^{N_\ba}$ are at least $2\delta$ apart.

	  We will build a round tree graph $A$ with vertex set $\{p_\ba^i: \ba \in \{0,1\}^*, 1 \leq i \leq N_{\ba}\}$ and prove that the map $\Pi:p_\ba^i\mapsto x_\ba^i$ determines a quasi-isometry $A \to W$. We first prove that $\Pi$ is coarsely onto.

	\begin{lemma}\label{lem:cantorarcs0}
		The sets $W$ and $\{x_\ba^i: \ba \in \{0,1\}^*, 1 \leq i \leq N_{\ba}\}$ are at bounded Hausdorff distance.
	\end{lemma}
	\begin{proof}
		Any $w \in W$ lies on a geodesic ray from $1$ to some $z \in J_\bbb$ for some $\bbb = (b_i) \in \{0,1\}^\N$.  
	Choose $n$ maximal so that $d(w,e) \geq \frac{-1}{\epsilon}\log(\beta^n)+C_1$.
	As $J_{\bbb}$ is $\beta^{n}$ close to $J_{b_1\cdots b_n}$, there exists some $x_{b_1\cdots b_n}^i$ so that $d(x_{b_1\cdots b_n}^i, w) \leq C_2=C_2(C,\epsilon,\delta,C_1)$.
		
		Conversely, given $x_\ba^i$ for $\ba \in \{0,1\}^n, 1 \leq i \leq N_\ba$, let $\bbb=(\ba,a_{n+1},\ldots) \in \{0,1\}^\N$.  Then as $J_\ba$ and $J_{\bbb}$ have Hausdorff distance $\leq 2\lambda\beta^n$, there exists some $z \in J_{\bbb}$ with $\rho(z_\ba^i, z) \leq 2\lambda\beta^n$.  Thus $d(x_\ba^i,w)$ is uniformly bounded where $w\in W$ is the point on the geodesic ray from $1$ to $z$ with $d(1,w)=\frac{-1}{\epsilon}\log(\beta^n)+C_1$.
	\end{proof}

	Now, each $z_{a_1\cdots a_{n}}^i$ projects by the $\beta^n$-following map to a point in some $J_{a_1\cdots a_{n-1}}[z_{a_1\cdots a_{n-1}}^j, z_{a_1\cdots a_{n-1}}^{j+1})$; denote this function $i\mapsto j$ by 
	\[ \pi_{a_1\cdots a_n}: \{1,\ldots, N_{a_1\cdots a_n}\} \to \{1,\ldots,N_{a_1\cdots a_{n-1}}\}. \]  
	We require $\pi_{a_1\cdots a_n}$ to map endpoints to endpoints, i.e.\ 
	$\pi_{a_1\cdots a_n}(1)=1$ and
	$\pi_{a_1\cdots a_n}(N_{a_1\cdots a_n})=N_{a_1\cdots a_{n-1}}$.

    We can now define the edges of $A$. The set of horizontal edges of $A$ is
    \[
        \{p_\ba^ip_\ba^{i+1}: \ba \in \{0,1\}^*, 1 \leq i < N_{\ba}\}.
    \]
    For each non-empty $(a_1,\ldots,a_n) \in \{0,1\}^*$, and each $j \in \{1,\ldots, N_{a_1\cdots a_{n-1}}\}$, if $i$ is minimal with $\pi_\ba(i)=j$ and $i'$ is minimal with $\pi_\ba(i')=j+1$, connect each of $p_\ba^i,\ldots,p_\ba^{i'-1}$ to $p_{a_1\cdots a_{n-1}}^j$ by a vertical edge.

	The next lemma ensures that $\Pi$ is Lipschitz on horizontal edges of $A$.
    
	\begin{lemma}\label{lem:cantorarcs1} For all non-empty $\ba \in \{0,1\}^*$ and $1 \leq i< N_\ba$, we have $\diam J_\ba[z_\ba^i, z_\ba^{i+1}] \in [\beta^n/4\lambda,\beta^n/2]$.  In particular, $\rho(z_\ba^i, z_\ba^{i+1}) \in [\beta^n/4\lambda,\beta^n/2]$.
	\end{lemma}
	\begin{proof}
		The lower bound is immediate since $\{z_\ba^i\}$ is a $\beta^n/4\lambda$-separated net.
		If the upper bound fails, there exists $\ba$, $k$ and $a,b \in J_\ba[z_\ba^k,z_\ba^{k+1}]$ so that $\rho(a,b) > \beta^n/2$.  Thus, there exists $p \in J_\ba[a,b]$ with $\rho(p, \{a,b\}) > \beta^n/4$.
		Since $\{z_\ba^i\}$ is a maximal $\beta^n/4\lambda$-separated net in $J_\ba$, there must exists $k'$ so that $\rho(p,z_\ba^{k'}) \leq \beta^n/4\lambda$.
		The arc $J_\ba[p, z_\ba^{k'}]$ passes through $p$ and either $a$ or $b$, but has diameter $\leq \beta^n/4$ by the quasi-arc condition, giving a contradiction.
	\end{proof}

	In order to control the vertical edges, we need more information about the functions $\pi_\ba$.
	\begin{lemma}\label{lem:cantorarcs2} Assuming $\beta<1/10\lambda^2$, for all non-empty $\ba \in \{0,1\}^*$ and $1 \leq i < N_\ba$,
	$|\pi_\ba(i)-\pi_\ba(i+1)| \leq 1$.  In particular, each $\pi_\ba$ is surjective.
	\end{lemma}
	In the following it is convenient to write $\ba=(a_1,\ldots,a_n)=(\ba',a_n)$, meaning $\ba'=(a_1,\ldots,a_{n-1})$.
	\begin{proof}
		If not, there exists $\ba=(a_1,\ldots,a_n)= (\ba',a_n)$, $1 \leq i < N_\ba$ and $1 \leq j < N_{\ba'}$ so that the $\beta^n$-following map projects 
		$z_\ba^i$ to a point $p \in J_{\ba'}[ z_{\ba'}^1, z_{\ba'}^j]$ and
		$z_\ba^{i+1}$ to a point $q \in J_{\ba'}[ z_{\ba'}^{j+1}, z_{\ba'}^{N_{\ba'}}]$,
		or vice-versa.  We assume the former; in the latter case, swap the roles of $z_\ba^i$ and $z_\ba^{i+1}$.
		In particular, since $p \leq z_{\ba'}^j < z_{\ba'}^{j+1} \leq q$ in the order on $J_{\ba'}$, $\diam J_{\ba'}[p,q] \geq \beta^{n-1}/4\lambda$.

		But on the other hand, by Lemma~\ref{lem:cantorarcs1} we have 
		\[ \rho(p,q) \leq \rho(p, z_\ba^i)+ \rho(z_\ba^i,z_\ba^{i+1})+\rho(z_\ba^{i+1},q) \leq \frac{5}{2}\beta^n < \beta^{n-1}, \]
		hence
		$\diam J_{\ba'}[p,q] \leq \frac{5}{2}\lambda \beta^n$
		which gives a contradiction if $\beta < 1/10\lambda^2$.
	\end{proof}

        The final lemma proves bounds on the horizontal branching.
        
	\begin{lemma}\label{lem:cantorarcs3}
		Assume $\beta < 1/12 \lambda^2$.  Then there exists $H\geq 2$ so that for any $(a_1,\ldots,a_n)\in \{0,1\}^*$ and any $1 \leq j < N_{a_1\cdots a_{n-1}}$, if $i$ is minimal with $\pi_{a_1\cdots a_n}(i)=j$ and $i'$ is minimal with $\pi_{a_1\cdots a_n}(i')=j+1$, then $i+2 \leq i' \leq i+H$.
	\end{lemma}
	\begin{proof}
		By Lemma~\ref{lem:cantorarcs2}, the only way ``$i+2 \leq i'$'' can fail is if for some $\ba=(a_1,\ldots, a_n)=(\ba',a_n)$ and $1 \leq j < N_{\ba'}$, if $i$ is minimal with $\pi_\ba(i)=j$, then $\pi_\ba(i+1)=j+1$.
		But then by minimality $\pi_\ba(i-1) < j$ (or $i=j=1$) and so the following map sends $z_\ba^{i-1}, z_\ba^{i+1}$ (or $z_\ba^1, z_\ba^2$) to points $p,q$ which satisfy  
		$\diam J_{\ba'}[p,q] \geq \beta^{n-1}/4\lambda$.
		But on the other hand 
		\[ \diam J_{\ba'}[p,q] \leq \lambda \rho(p,q) 
		\leq \lambda \bigg(\beta^n + \frac{1}{2}\beta^n + \frac{1}{2}\beta^n+\beta^n\bigg) = 3\lambda \beta^n \]
		by the quasi-arc property, a contradiction for $\beta < 1/12 \lambda^2$.

		We now show under the hypotheses of the Lemma that $i' \leq i+H$ for a uniform $H$.
		Note that $\pi_\ba(i'-1)=j$ by minimality and Lemma~\ref{lem:cantorarcs2}.
		Now by the $\beta^n$-following map, every point $z_\ba^i, z_\ba^{i+1},\ldots, z_\ba^{i'-1}$ is $\beta^n$ close to the arc $J_{\ba'}[z_{\ba'}^j, z_{\ba'}^{j+1}]$, which has diameter $\leq \beta^{n-1}/2$ by Lemma~\ref{lem:cantorarcs1}.
		Since the points $\{z_\ba^k : i \leq k < i'\}$ are $\beta^n/4\lambda$ separated, the doubling property for $\bdry G$ gives a uniform bound on the cardinality $i'-i$ of this set.
	\end{proof}

		By Lemma~\ref{lem:cantorarcs3} and construction, $A$ satisfies the definition of bounded irregular round tree with constant vertical branching $2$ and horizontal branching in $[2,H]$.  It remains to observe that $A$ with its graph metric is quasi-isometric to the subset $W \subset G$.
	
		\begin{proof}[Proof of Proposition \ref{prop:QI-round-tree-graph}]
		Recall $\{x_\ba^i : \ba \in \{0,1\}^*, 1 \leq i \leq N_\ba\}$ is coarsely dense in $W$ by Lemma \ref{lem:cantorarcs0}.
		The function $\Pi$ is a bijection onto its image, so it suffices to show that $\Pi$ and $\Pi^{-1}$ (restricted to its image) are coarsely Lipschitz.

		To show $\Pi$ is coarsely Lipschitz, since $A$ is a graph it suffices to bound the images of endpoints of edges.
		For a horizontal edge between $p_\ba^i$ and $p_\ba^{i+1}$ for some $\ba \in \{0,1\}^*, 1 \leq i < N_\ba$, by Lemma~\ref{lem:cantorarcs1} we have $\rho(z_\ba^i, z_\ba^{i+1}) \leq \beta^n/2$, hence the Gromov product $(z_\ba^i | z_\ba^{i+1})_1 \geq \frac{-1}{\epsilon}\log(\beta^n)-C$ for some $C$.  As each $x^j_\ba$ is a point $-\frac1\varepsilon\log(\beta^n)+C_1$ from $1$ along a geodesic to $z^j_\ba$, we deduce that $d(x_\ba^i, x_\ba^{i+1})$ is uniformly bounded by hyperbolicity.
		For a vertical edge between $p_\ba^k$ and $p_{\ba'}^j$, with $\ba=(\ba',a_n) \in \{0,1\}^n$, $1 \leq i\leq k \leq i'-1 \leq N_\ba$ and $i,i'$ minimal with $\pi_\ba(i)=j$ and $\pi_\ba(i')=j+1$ respectively, we have by Lemma~\ref{lem:cantorarcs3} (and proof) that $\rho(z_\ba^k, z_{\ba'}^j) \leq \beta^{n-1}$, so again the Gromov product satisfies $(z_\ba^k| z_{\ba'}^j)_1 \geq \frac{-1}{\epsilon}\log(\beta^n)-C'$, and $d(x_\ba^k, x_{\ba'}^j)$ is uniformly bounded by hyperbolicity.

		To show $\Pi^{-1}$ is coarsely Lipschitz, since $W$ is quasi-convex in $G$, $W$ is coarsely geodesic, and so it suffices to show that for any $M \geq 0$ there exists $M' < \infty$ so that $d(x_\ba^i, x_{\bbb}^j) \leq M$ implies $d_A(p_{\ba}^i,p_{\bbb}^j) \leq M'$.
		Assume we have such $x_\ba^i, x_\bbb^j$, with $\ba \in \{0,1\}^m, \bbb \in \{0,1\}^n$.
		Then $|n-m|$ is uniformly bounded depending on $M$; note that $\rho(z_\ba^i, z_\bbb^j) \preceq \beta^{m} \asymp \beta^{n}$.
		By Lemma~\ref{lem:cantorarcs0}, there exist $\hat\ba = (\ba, a_{m+1},\ldots) , \hat\bbb = (\bbb, b_{n+1},\ldots)\in \{0,1\}^\N$ with $\rho(z_\ba^i, J_{\hat\ba}) \preceq \beta^{m}$ and $\rho(z_\bbb^j, J_{\hat\bbb}) \preceq \beta^{n}$.
		Thus the minimum distance $\rho(J_{\hat\ba}, J_{\hat\bbb}) \preceq \beta^m$,
		and so if $k = \min\{l:a_l\neq b_l\}$, we have $\beta^k \preceq \beta^m$, which implies that $k \geq m-C$ for some constant $C$.  Hence for $\bc=(a_1,\ldots,a_k)=(b_1,\ldots,b_k)$, there exists $i', j'$ so that $d_A(p_\ba^i,p_\bc^{i'}), d_A(p_\bbb^j,p_\bc^{j'}) \leq C$ and $\rho(z_\ba^i,z_\bc^{i'}), \rho(z_\bbb^j, z_\bc^{j'}) \leq 2 \beta^{k+1}$.  Then $\rho(z_\bc^{i'}, z_\bc^{j'}) \preceq 4\beta^{k+1}+\beta^m \preceq \beta^k$.
		Using that $J_\bc$ is a $\lambda'$-quasi-arc, the chain of $\beta^k/4\lambda$-separated points $\{z_\bc^{k}: k \in [i',j']\} \subset B(z_\bc^{i'},C\lambda'\beta^k)$, and so $|j'-i'|$ is uniformly bounded by doubling.
		Hence $d_A(p_\bc^{i'},p_\bc^{j'})$ is uniformly bounded, and we are done.
	\end{proof}

	Having shown $W$ is quasi-isometric to a round tree graph with vertical branching $2$ and horizontal branching in $[2,H]$, the proposition follows from Proposition~\ref{prop:roundtree-lower}.
\end{proof}

\section{Groups with logarithmic separation}\label{sec:logupper}
\label{sec:all-QH-vertex-groups}

We will prove Theorem~\ref{thm:logsep} using the separation profile of Benjamini-Schramm-Tim\'ar~\cite{BenSchTim-12-separation-graphs}. We recall the definition and key properties for convenience.

\begin{definition}\label{defn:sep}
    Let $\Gamma$ be an $r$-vertex graph and let $\varepsilon\in(0,1)$. The $\varepsilon$-cut size of $\Gamma$, $\cut^\varepsilon(\Gamma)$ is the minimum cardinality of a subset $S\subseteq V(\Gamma)$ so that every connected component of $\Gamma-S$ has at most $\varepsilon r$ vertices.

    The $\varepsilon$ separation profile of a graph $X$ is the function
    \[
		\sep^\varepsilon_X(r) := \max\setcon{\cut^\varepsilon(\Gamma)}{\Gamma\leq X,\ |\Gamma|\leq r}.
    \]
\end{definition}

\begin{proposition}[\cite{BenSchTim-12-separation-graphs,HumeMackTess-Pprof}]
    Let $X$ be a bounded degree graph. Then for any $\varepsilon \in (0, 1)$,
    \[
        \Lambda^1_X(r) \simeq \sep^\varepsilon_X(r).
    \]
\end{proposition}

Let us recall the statement of Theorem \ref{thm:logsep}.  In this section we will denote $\sep^{\frac23}$ by $\sep$ and $\cut^{\frac23}$ by $\cut$ (cf.\ Definition \ref{defn:sep}).

\begin{varthm}[Theorem~\ref{thm:logsep}.]
	Let $G$ be a one-ended hyperbolic group which has a JSJ decomposition over $2$-ended subgroups consisting only of quadratically hanging and cylindrical vertex groups.
	Then $\sep_{G}(r) \simeq \log(r)$.
\end{varthm}

By Shepherd--Woodhouse~\cite[Theorem 1.2; see Remark 3.6, Theorem 3.9, Lemma 3.12]{ShepWood-22-QI-rigid-vfree-amalgams}, we can find a finite index subgroup which has a decomposition as a graph of groups with all vertex groups fundamental groups of either (1) orientable hyperbolic surfaces or (2) of circles, and all edge groups infinite cyclic which identify a vertex group of type (2) with a subgroup of a vertex group of type (1) corresponding to a boundary circle.  
This reduces Theorem~\ref{thm:logsep} to the following more geometric theorem, which we prove in this section.

\begin{theorem}\label{thm:logsep-geom}
Let $\cC$ be a finite collection of compact hyperbolic surfaces with totally geodesic boundary.
Let $X$ be a connected space obtained by gluing together the surfaces in $\cC$, where each gluing identifies some set of $\geq 2$ isometric boundary curves of surfaces in $\cC$.
	Then $G= \pi_1(X)$ satisfies $\sep_{G}(r)\lesssim \log(r)$.
	In particular, if $G$ is one-ended, $\sep_G(r) \simeq \log(r)$.
\end{theorem}
Note that such $X$ is $\CAT(-1)$.

The idea behind the proof is as follows: we work in a standard graph approximation $N=(V(N),E(N))$ to the universal cover $U:=\widetilde{X}$, where $V(N)$ is a maximal $1$-separated subset of $U$ and
\[
    E(N) = \setcon{nn'}{d_U(n,n')\leq 3}.
\]
It is a standard result that $N$ (with its shortest path metric) is quasi-isometric to $U$ (which is itself quasi-isometric to $G$) so it suffices to prove that $\sep_N(r) \simeq \log(r)$.

Fix a subset $Z\subset V(N)$ of cardinality $r$. The strategy is as follows. We may disconnect the space $U$ into convex connected components by removing a tree $C$ rooted at a ``centrepoint'' $z$ of $Z$ built from a uniformly bounded number of geodesics. As each connected component has non-abelian free fundamental group there are enough ways of doing this that most of them will intersect the $3$-neighbourhood of $Z$ only inside the ball of radius $\simeq\log(r)$ centred at $z$. Choosing such a $C$, we obtain a cutset $Z\cap N_3(C)$ for $Z$ of size $\lesssim \log(r)$. A good choice of centrepoint ensures that this cutset is a $\frac23$-cutset. We will break the proof of Theorem \ref{thm:logsep-geom} into several steps following this rough outline.

\subsection*{Step 1: finding rays that avoid lifts of attaching curves}

\begin{definition} Let $U$ be a $\CAT(0)$ space, let $o\in U$ and let $\gamma$ be a geodesic of positive length starting at $o$. Given $\theta\in[0,\pi/2)$, define the \textbf{$\theta$-wedge centred at $o$ with direction $\gamma$} to be the set of all points $z\neq o$ such that
\[
 \angle_o(\gamma,[o,z]) < \theta,
\]
where $[o,z]$ is the unique geodesic from $o$ to $z$ in $X$.
\end{definition}

\begin{definition} Let $U$ be a $\CAT(0)$ space, and $A \subset U$. Given $z\in U$, define $\mathcal R^A_z$ to be the set of all geodesic rays $\gamma$ with $\gamma(0)=z$ such that $\gamma\setminus\{z\}$ is disjoint from $A$. Define $R^A_z$ to be the set of all points lying on geodesic rays in $\mathcal R^A_z$.
\end{definition}

We use the following simple fact from hyperbolic geometry.
\begin{lemma}
	\label{lem:hyp-visual-angle}
	There exists $C_v >0$ so that for any point $o \in \HH^2$ and bi-infinite geodesic $\gamma \subset \HH^2$ with $d(o,\gamma) \geq C_v$, then the visual angle at $o$ between points in $\gamma$ is at most $\frac{\pi}{20}$.
\end{lemma}

We now consider the space $U = \widetilde{X}$ of Theorem~\ref{thm:logsep-geom}, and the subset $A \subset U$ of lifts of attaching closed curves. As $A$ is now fixed, we will write $\mathcal R_z$ for $\mathcal R_z^A$ and $R_z$ for $R_z^A$.

We can and do deform the hyperbolic surfaces in $X$ by shrinking the attaching circles, so that the distances between attaching curves are all $\geq 4C_v$.

\begin{lemma}\label{lem:expmany} There exists $\lambda>0$ with the following property.  For every $\frac{\pi}{8}$-wedge $W$ in $U$, denoting the centre of $W$ by $z$ and its direction by $\gamma_W$, if  $d_U(z,A\cap W)\geq 2C_v$ then for every $R$ there
exists $A_R\subset\mathcal R_z$ so that:
\begin{itemize}
    \item for each $\gamma\in A_R$, $\gamma\subset W$,
    \item $|A_R|\geq \lambda\exp(\lambda R)$
    \item for each distinct pair $\gamma,\gamma'\in A_R$,
    \[
     N_3(\gamma)\cap N_3(\gamma') \subseteq B(z;R).
    \]
\end{itemize}
\end{lemma}
\begin{proof}
    The point $z$ lies in a lift $Y$ of the closure of one of the components of $X$ minus the set of attaching curves. As this component has a non-abelian free group as its fundamental group and hyperbolic metric, $Y$ is isometric to a convex subset of $\HH^2$ which is quasi-isometric (with constants that may be chosen independently of $z$) to a $3$-regular tree.  

	If $\beta \subset A$ is a lift of an attaching curve, then $d(z,\beta\cap W)\geq 2C_v$.  The maximal possible visible angle from $z$ of $\beta \cap W$ is $\frac{\pi}{10}$.  Indeed the worst case is when $\beta$ meets a boundary ray of $W$ orthogonally at distance $2C_v$ from $z$, so the claim follows from Lemma~\ref{lem:hyp-visual-angle}. 
	
	Using this fact, we may choose a ray $\gamma\in R_z$ with $\gamma\subseteq Y$ satisfying $\angle_z(\gamma,\gamma_W)\leq \frac{\pi}{10}$. For some uniform constant $K$ (not depending on $W$), the connected component of $Y\setminus B(z;K)$ containing $\gamma\setminus B(z;K)$ (which we denote by $Y'$) is also contained in $W$. Now, $Y'$ is itself uniformly quasi-isometric to a $3$-regular tree and the result follows.
\end{proof}

\subsection*{Step 2: finding a centre point and building a cut set}
Recall that $Z\subset V(N)$ has cardinality $r$. We mainly consider $Z$ as a subset of $U$.

\begin{lemma}[Centrepoint lemma] There is some $z\in U$ so that for any convex subset $C$ of $U$ which does not contain $z$, we have $|Z\cap C|\leq \frac23|Z|$.
\end{lemma}
\begin{proof}
	By the pigeonhole principle, the intersection of any three convex subsets of $U$ which contain strictly more than $\frac23|Z|$ of the elements in $Z$ is non-empty. Therefore, as $U$ is $\CAT(0)$ and has compact topological dimension $2$, the intersection of all convex subsets of $U$ which contain strictly more than $\frac23|Z|$ of the elements in $Z$ is non-empty \cite[Proposition 5.3]{Kleiner-length-space}. Any $z$ in this set suffices.
\end{proof}

Now we construct the cut set for $Z$. Let $\gamma$ be a geodesic from $z$ to $A$ (the length of this geodesic may be bounded independent of $z$ as the connected components of $U\setminus A$ are uniformly quasi-isometric to $3$-regular trees), denote the end of $\gamma$ on $A$ by $o$ and let $\gamma'$ be the connected component of $A$ containing the end point of $\gamma$ on $A$. Let $\gamma^+$ be a short geodesic continuation of $\gamma$ starting at $z$ and moving away from $A$ (if $z\not\in A$) or a short geodesic orthogonal to the lift of the attaching curve if $z\in A$. Let $\gamma'_0,\gamma'_1$ be short geodesics starting at $z$ making angle $\pi/4$ anticlockwise and clockwise with $\gamma_+$ respectively. Let $W_0,W_1$ be open $\pi/8$ wedges centred at $z$ with directions $\gamma'_0,\gamma'_1$. As different lifts of attaching curves are $4C_v$-separated, and $W_0\cap \gamma'=W_1\cap \gamma'=\emptyset$ by construction, $d_U(z,A\cap W_i)\geq 2C_v$ for $i=0,1$.

Applying Lemma \ref{lem:expmany} with $R=\frac{1}{\lambda}\log(\frac{1+r}{\lambda})$ we obtain two collections of $\geq r+1$ rays $A_i\subset\mathcal R_z$ where each ray in $A_i$ is contained in $W_i$ for $i=0,1$.  By the lemma and the pigeonhole principle, there is some $\gamma_i\in A_i$ such that $N_3(\gamma_i)\cap Z\subset B(z,R)$. Set $C_0=\gamma\cup\gamma_0\cup\gamma_1$. We now repeat this argument in every other connected component of $U\setminus\gamma'$ (replacing $z$ with $o$) to obtain sets $C_1,\ldots,C_k$. Define $C=\bigcup C_i$. We note that, by construction, $N_3(C)\cap Z$ is contained in $B(z,R)$. Since $C$ is a union of at most $3k+3$ geodesics of length $\leq R$, $|N_3(C)\cap Z|\leq K'R \leq K\log(1+r)$ for some constant $K=K(X)$ which is independent of $r$ and $Z$.

\begin{figure}[h]
    \centering
    \begin{tikzpicture}[xscale=0.72,yscale=0.72] 
    \draw[thick, dotted] (-5,0) -- (5,0);
        \begin{scope}
            \clip (0,0) circle (5cm);
            \draw[thick, dotted] (-5,0) -- (5,0);
            \fill[black!20] (0,1) -- (2,5) -- (5,3.8) -- (0,1);
            \fill[black!20] (0,1) -- (-2,5) -- (-5,3.8) -- (0,1);
            \fill[black!20] (0,0) -- (-1.91,-4.62) -- (-4.62,-4.62) -- (-4.62,-1.91) -- (0,0);
            \fill[black!20] (0,0) -- (1.91,-4.62) -- (4.62,-4.62) -- (4.62,-1.91) -- (0,0);
            \draw[very thick] (0,0) -- (6,-5.6);
            \draw[very thick] (0,0) -- (-3,-4.8);
            \draw[very thick] (0,1) -- (2.3,5);
            \draw[very thick] (0,1) -- (-5,5);
            \draw[thick, dotted] (0,-5) circle (1.8cm);
            \draw[thick, dotted] (0,5) circle (1.8cm);
            \draw[thick, dotted] (4.33,2.5) circle (1cm);
            \draw[thick, dotted] (-4.33,2.5) circle (1cm);
            \draw[thick, dotted] (4.33,-2.5) circle (1cm);
            \draw[thick, dotted] (-4.33,-2.5) circle (1cm);
        \end{scope}
    \draw (0,0) circle (5cm);
    \fill (0,0) circle (2pt);
    \fill (0,1) circle (2pt);
    \draw[very thick] (0,0) -- (0,1);
    \node[left] at (-5,0) {$\gamma'$};
    \node[below] at (0,-0.1) {$o$};
    \node[right] at (0,1) {$z$};
    \node[left] at (0,0.3) {$\gamma$};
    \node[] at (-2.5,3.6) {$W_0$};
    \node[] at (2.5,3.6) {$W_1$};
    \node[] at (-0.9,2) {$\gamma_0$};
    \node[] at (1,2) {$\gamma_1$};

    \draw[dashed] (0,1) circle (5mm)
                    (9,1.5) ++(110:2.5) -- (0,1.5)
                    (9,1.5) ++(-110:2.5) -- (0,0.5)
                    (9,1.5) circle (2.5cm);

    \fill[black!20] (9,1.5) -- +(22.5:2.5) arc (22.5:67.5:2.5) -- (9,1.5)
                    (9,1.5) -- +(112.5:2.5) arc (112.5:157.5:2.5) -- (9,1.5);

    \fill (9,1.5) circle (2pt);
    \draw[] (9,-1) -- (9,2.75)
            (8.85,1.35) rectangle (9.15,1.5);
            (9,1.5) -- +(45:1.25)
            (9,1.5) -- +(135:1.25)
            (9.5,1.5) arc (0:180.5:0.5);
    \foreach \j in {0,1,2,3,4,5,6,7,8}{
    \draw[] (9,1.5) -- +(22.5*\j:0.7);
    }

\begin{scope}[xshift=9cm, yshift=1.5cm]
    \node[below right] at (0,0) {$z$};
    \node[] at (45:1.7) {$\gamma'_1$};
    \node[] at (90:1.7) {$\gamma^+$};
    \node[] at (135:1.7) {$\gamma'_0$};

    \foreach \i in {1,3,5,7,9,11,13,15}{
    \node[] at (11.25*\i:0.7) {$\ast$};
    }
\end{scope}    

    \node[] at (9,-1.5) {$\ast=\frac\pi8$};
                    
    \end{tikzpicture}
	\caption{Constructing $C$ (the bold lines) in two of the components of $U\setminus\gamma'$}
    \label{fig:logcutset}
\end{figure}
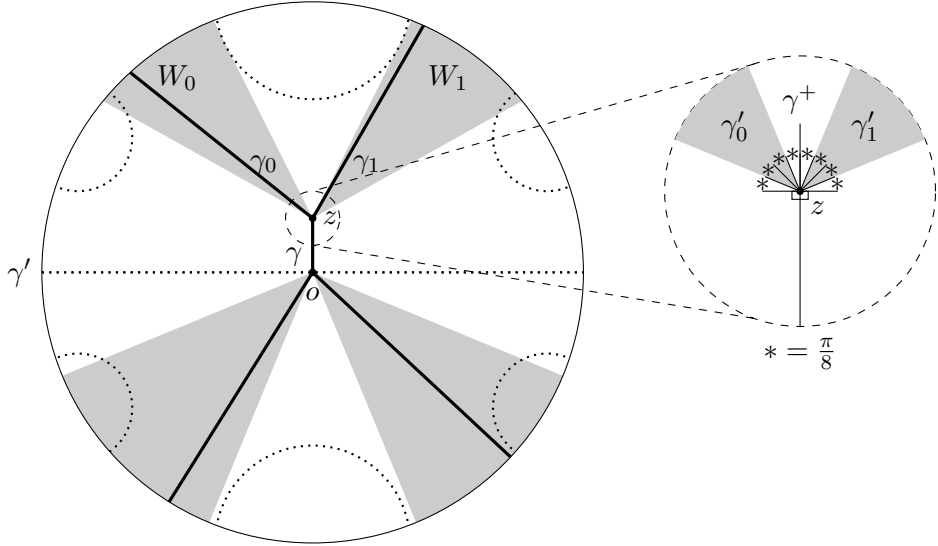

\subsection*{Step 3: components of $U\setminus C$ are convex}

Let $Y$ be a connected component of $U\setminus C$ and let $x,x'\in Y$. By considering the connected component(s) of $Y\setminus \gamma'$ containing $x$ and $x'$, we find a subset $Y'$ of $Y$ of one of the three forms shown in Figure \ref{fig:compsU-C}. 

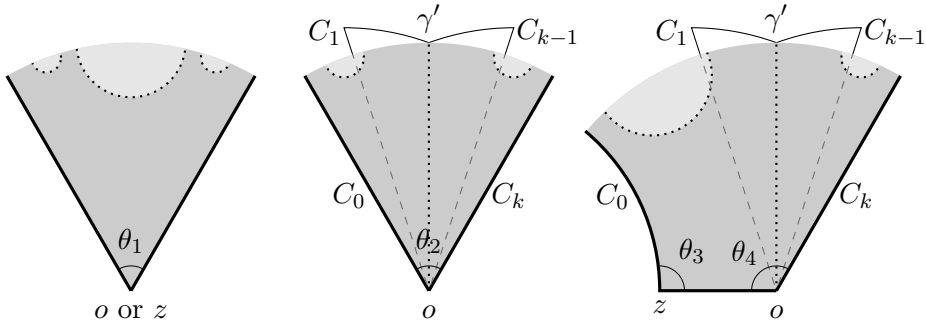
\begin{figure}[h]
    \centering
    \begin{tikzpicture}[xscale=0.66, yscale=0.66]
    \fill[black!20] (0,0) -- (120:5cm) arc (120:60:5cm) -- (0,0);
    \draw[very thick] (0,0) -- (120:5cm);
    \draw[very thick] (0,0) -- (60:5cm);
    \node[below] at (0,-0.1) {$o$ or $z$};
    \begin{scope}
    \clip (0,0) -- (120:5cm) arc (120:60:5cm) -- (0,0);
    \filldraw[dotted, thick, fill=black!10] (0,5) circle (1.1cm);
    \filldraw[dotted, thick, fill=black!10] (1.71,4.70) circle (0.3cm);
    \filldraw[dotted, thick, fill=black!10] (-1.71,4.70) circle (0.3cm);    
    \end{scope}
    \draw[very thin] (120:0.5) arc (120:60:0.5);
    \node[] at (0,1) {$\theta_1$};

    \begin{scope}[xshift=6cm]
    \draw[] (0,5) arc (70:90:5cm) -- (0,0);
    \draw[] (0,5) arc (110:90:5cm) -- (0,0);
    \fill[black!20] (0,0) -- (120:5cm) arc (120:60:5cm) -- (0,0);
    \draw[very thick] (0,0) -- (120:5cm);
    \draw[very thick] (0,0) -- (60:5cm);
    \node[below] at (0,-0.1) {$o$};
    \begin{scope}
    \clip (0,0) -- (120:5cm) arc (120:60:5cm) -- (0,0);
    \draw[dotted, thick] (0,0) -- (0,5);
    \filldraw[dotted, thick, fill=black!10] (1.71,4.70) circle (0.4cm);
    \filldraw[dotted, thick, fill=black!10] (-1.71,4.70) circle (0.4cm);
    \draw[dashed, very thin, black!60] (0,5) arc (70:90:5cm) -- (0,0);
    \draw[dashed, very thin, black!60] (0,5) arc (110:90:5cm) -- (0,0);
    \end{scope}
    \node[above] at (0,5) {$\gamma'$};
    \node[] at (130:2.5) {$C_0$};
    \node[] at (-2.1,5.2) {$C_1$};
    \node[] at (2.4,5.2) {$C_{k-1}$};
    \node[] at (50:2.5) {$C_k$};
    \draw[very thin] (120:0.5) arc (120:60:0.5);
    \node[] at (0,1) {$\theta_2$};
    \end{scope}

    \begin{scope}[xshift=13cm, xscale=-1]
    \draw[] (0,5) arc (70:90:5cm) -- (0,0);
    \draw[] (0,5) arc (110:90:5cm) -- (0,0);
    \fill[black!20] (0,0) -- (120:5cm) arc (120:40:5cm) arc (130:180:4.2cm) -- (0,0);
    \draw[very thick] (0,0) -- (120:5cm);
    \draw[very thick] (0,0) -- (2.35,0) arc (180:130:4.2cm);
    \node[below] at (0,-0.1) {$o$};
    \begin{scope}
    \clip (0,0) -- (120:5cm) arc (120:40:5cm) arc (130:180:4.2cm) -- (0,0);
    \draw[dotted, thick] (0,0) -- (0,5);
    \filldraw[dotted, thick, fill=black!10] (2.5,4.33) circle (1.2cm);
    \filldraw[dotted, thick, fill=black!10] (-1.71,4.70) circle (0.4cm);    
    \draw[dashed, very thin, black!60] (0,5) arc (70:90:5cm) -- (0,0);
    \draw[dashed, very thin, black!60] (0,5) arc (110:90:5cm) -- (0,0);
    \end{scope}
    \node[above] at (0,5) {$\gamma'$};
    \node[] at (3.3,1.9) {$C_0$};
    \node[] at (2.1,5.2) {$C_1$};
    \node[] at (-2.4,5.2) {$C_{k-1}$};
    \node[] at (130:2.5) {$C_k$};
    \node[below] at (2.35,0) {$z$};
    \draw[very thin] (120:0.5) arc (120:0:0.5);
    \node[] at (50:1) {$\theta_4$};
    \draw[very thin] (1.85,0) arc (180:90:0.5);
    \node[] at (1.7,0.8) {$\theta_3$};
    \end{scope}
    \end{tikzpicture}
    \caption{Three types of component of $U\setminus C$, with $Y'\setminus Y''$ highlighted in pale grey and $Y''$ in a darker grey}
    \label{fig:compsU-C}
\end{figure}

It suffices to show that each such set is convex. Let $x,x'\in Y'$ and define $Y''$ to be the connected component of $Y'\setminus (A\setminus\gamma')$ with $z$ and/or $o$ in its boundary. Note that $Y''$ is convex as it is isometric to an intersection of halfspaces in $\HH^2$, this relies on the fact that all the internal angles (marked as $\theta_1,\ldots,\theta_4$ in Figure \ref{fig:compsU-C}) are between $\frac\pi4$ and $\frac{7\pi}8$. If $[x,x']$ does not intersect $Y''$ then it is contained in $Y$ as $Y''$ separates $Y\setminus Y''$ from $U\setminus Y$. If it does intersect $Y''$, then as $Y''$ is convex, $[x,x']\setminus (Y'\setminus Y'')$ is contained in $Y''$. As $Y'\cup Y''\subset Y$, $[x,x']\subset Y$, so $Y$ is convex.

\subsection*{Step 4: completing the proof}

We claim that $N_3(C)\cap Z \subset V(N)$ is a $\frac23$-cutset of $Z$ (considered as a subgraph of $N$); that is, every connected component of $Z \setminus (N_3(C)\cap Z) \subset Z$ has $\leq \frac{2}{3}|Z|$ points in each connected component.

If $z',z''\in Z$ are in different connected components of $U\setminus N_3(C)$, then any path $z'=z_1,\ldots,z_m=z''$ must contain a point in $N_3(C)$, as $d_U(z_i,z_{i+1})\leq 3$ for every edge $z_iz_{i+1}$ in $N$. Therefore, each connected component of $Z\setminus N_3(C)$ is contained in a connected component of $U\setminus C$, so by Helly's lemma, contains at most $\frac23|Z|$ vertices. We also have $|N_3(C)\cap Z|\leq K\log(1+r)$ by Step 2. As these bounds are independent of $Z$ and $r$, we see that
\[
	\sep^{\frac23}_N(r) \leq K\log(1+r)
\]
as required. \qed

\bibliographystyle{alpha}
\bibliography{biblio}

\end{document}